%% file: kpieri.tex
\documentclass{amsart}
\usepackage{amsmath, amsfonts, amssymb}
\usepackage{graphicx}
\usepackage{psfrag}
\usepackage{tableau}

\input{preamble}
\input{psfrag}

\begin{document}

\title{Pieri rules for the $K$-theory of cominuscule Grassmannians}

\date{May 14, 2010}

\author{Anders Skovsted Buch}
\address{Department of Mathematics, Rutgers University, 110
  Frelinghuysen Road, Piscataway, NJ 08854, USA}
\email{asbuch@math.rutgers.edu\vspace{-2mm}}

\author{Vijay Ravikumar}
\email{vravikum@math.rutgers.edu}


\thanks{The first author was supported in part by NSF Grant
  DMS-0603822.}

\begin{abstract}
  We prove Pieri formulas for the multiplication with special Schubert
  classes in the $K$-theory of all cominuscule Grassmannians.  For
  Grassmannians of type A this gives a new proof of a formula of
  Lenart.  Our formula is new for Lagrangian Grassmannians, and for
  orthogonal Grassmannians it proves a special case of a conjectural
  Littlewood-Richardson rule of Thomas and Yong.  Recent work of
  Thomas and Yong and of E.~Clifford has shown that the full
  Littlewood-Richardson rule for orthogonal Grassmannians follows from
  the Pieri case proved here.  We describe the $K$-theoretic Pieri
  coefficients both as integers determined by positive recursive
  identities and as the number of certain tableaux.  The proof is
  based on a computation of the sheaf Euler characteristic of triple
  intersections of Schubert varieties, where at least one Schubert
  variety is special.
\end{abstract}

\maketitle

\section{Introduction}\label{sec:intro}

By a {\em cominuscule Grassmannian\/} we will mean a Grassmann variety
$\Gr(m,n)$ of type A, a Lagrangian Grassmannian $\LG(n,2n)$, or a
maximal orthogonal Grassmannian $\OG(n,2n+1) \cong \OG(n+1,2n+2)$.
The goal of this paper is to prove (old and new) Pieri rules for the
multiplication by special Schubert classes in the Grothendieck ring of
any cominuscule Grassmannian.

Any homogeneous space $X$ has a decomposition into Schubert varieties
$X^\la$, which for cominuscule Grassmannians are indexed by partitions
$\la = (\la_1 \geq \dots \geq \la_\ell \geq 0)$ in a way so that the
codimension of $X^\la$ in $X$ is equal to the weight $|\la| = \sum
\la_i$.  The Schubert classes $[X^\la] \in H^{2|\la|}(X;\Z)$ defined
by the Schubert varieties give a $\Z$-basis for the singular
cohomology ring $H^*(X) = H^*(X;\Z)$.  This ring is furthermore
generated as a $\Z$-algebra by the special Schubert classes $[X^p]$
given by partitions with a single part $p$.  The Littlewood-Richardson
coefficients $c^\nu_{\la\mu}$ are the structure constants for $H^*(X)$
with respect to the Schubert basis, i.e.\ they are determined by the
identity
\begin{equation} 
  [X^\la] \cdot [X^\mu] = \sum_{|\nu|=|\la|+|\mu|} c^\nu_{\la\mu}\, [X^\nu]
\end{equation}
in $H^*(X)$.  Each coefficient $c^\nu_{\la\mu}$ in this sum depends on
the type of the Grassmannians $X$ as well as the partitions.  It is
equal to the number of points in the intersection of general
translates of the Schubert varieties $X^\la$, $X^\mu$, and
$X^{\nu^\vee}$, where $\nu^\vee$ is the Poincare dual partition of
$\nu$.  This gives the identity
\begin{equation}\label{eqn:ctrip}
  c^\nu_{\la\mu} = \int_X [X^\la] \cdot [X^\mu] \cdot [X^{\nu^\vee}] 
\end{equation}
when $|\nu|=|\la|+|\mu|$.  The Littlewood-Richardson coefficients of
cominuscule Grassmannians are well understood and are described by
variants of the celebrated Littlewood-Richardson rule, see e.g.\
\cite{macdonald:symmetric*1, fulton:young,
  leeuwen:littlewood-richardson, buch.kresch.ea:littlewood-richardson}
and the references given there.

The cohomology ring $H^*(X)$ is generalized by various other rings,
including the equivariant cohomology ring $H^*_T(X)$, the quantum
cohomology ring $\QH(X)$, and the Grothendieck ring $K(X)$ of
algebraic vector bundles on $X$, also called the $K$-theory of $X$.
Every Schubert variety $X^\la$ has a Grothendieck class $\cO^\la =
[\cO_{X^\la}]$ in $K(X)$ (see \S\ref{sec:ktheory}) and these classes
form a basis for the Grothendieck ring.  We can therefore define {\em
  $K$-theoretic Schubert structure constants\/} for $X$ by the
identity
\begin{equation}\label{eqn:kcoef}
  \cO^\la \cdot \cO^\mu = \sum_\nu c^\nu_{\la\mu}\, \cO^\nu \,.
\end{equation}
The coefficient $c^\nu_{\la\mu}$ is non-zero only if $|\nu| \geq
|\la|+|\mu|$, and for $|\nu|=|\la|+|\mu|$ it agrees with the
cohomological structure constant of the same name.  The $K$-theoretic
structure constants have signs that alternate with codimension, in the
sense that $(-1)^{|\nu|-|\la|-|\mu|}\, c^\nu_{\la\mu} \geq 0$; this
was proved by the first author for Grassmannians of type A
\cite{buch:littlewood-richardson} and by Brion for arbitrary
homogeneous spaces \cite{brion:positivity}.  If $X$ is a Grassmann
variety of type A, then several Littlewood-Richardson rules are
available that express the absolute value $|c^\nu_{\la\mu}|$ as the
number of elements in a combinatorially defined set
\cite{buch:littlewood-richardson, lascoux:transition,
  buch.kresch.ea:stable, thomas.yong:jeu}.  An earlier Pieri formula
of Lenart \cite{lenart:combinatorial} shows that the coefficients
$c^\nu_{\la,p}$, that describe multiplication with the special classes
$\cO^p$, are equal to signed binomial coefficients.  For Lagrangian
and orthogonal Grassmannians it is known how to multiply with the
Schubert divisor $\cO^1$, by using Lenart and Postnikov's
$K$-theoretic Chevalley formula which works for arbitrary homogeneous
spaces \cite{lenart.postnikov:affine}.  More recently Thomas and Yong
\cite{thomas.yong:jeu} have conjectured a full Littlewood-Richardson
rule for the $K$-theory of maximal orthogonal Grassmannians based on
$K$-theoretic jeu-de-taquin slides.  No such conjecture is presently
available for Lagrangian Grassmannians.

The main results of this paper are $K$-theoretic Pieri rules for
maximal orthogonal Grassmannians and Lagrangian Grassmannians.  These
rules are stated in two equivalent versions.  The first is a set of
recursive identities that make it possible to compute any coefficient
$c^\nu_{\la,p}$ from simpler ones in a way that makes the alternation
of signs transparent.  The second is a Littlewood-Richardson rule that
expresses $|c^\nu_{\la,p}|$ as the number of tableaux satisfying
certain properties.  For maximal orthogonal Grassmannians, Itai
Feigenbaum and Emily Sergel have shown us a proof that these tableaux
are identical to the increasing tableaux appearing in the conjecture
of Thomas and Yong \cite{thomas.yong:jeu}, which confirms that this
conjecture computes all Pieri coefficients correctly.  The tableaux
being counted for Lagrangian Grassmannians are new and contain both
primed and unprimed integers; it would be very interesting to extend
this rule to a full Littlewood-Richardson rule for all the
$K$-theoretic structure constants.  While the cohomological Schubert
calculus of Lagrangian and maximal orthogonal Grassmannians is
essentially the same, our results show that these spaces have quite
different $K$-theoretic Schubert calculus.  For example, we prove that
the $K$-theoretic structure constants $c^{\nu^\vee}_{\la\mu}$ for
orthogonal Grassmannians are invariant under permutations of $\la$,
$\mu$, and $\nu$, but show by example that this fails for Lagrangian
Grassmannians.

Thomas and Yong have conjectured a $K$-theoretic Littlewood-Richardson
rule for all minuscule homogeneous spaces, and proved that their rule
is a consequence of a well-definedness property of a jeu-de-taquin
algorithm, together with a proof that their conjecture gives the
correct answer for a generating set of Schubert classes
\cite{thomas.yong:jeu}.  For maximal orthogonal Grassmannians, the
well-definedness property has later been proved by Thomas and Yong
\cite{thomas.yong:k-theoretic} and by Edward Clifford
\cite{clifford:thesis}.  Since the special classes $\cO^p$ generate
the $K$-theory ring, the Pieri rule proved here provides the second
ingredient required for the proof of Thomas and Yong's conjecture.

The standard way to prove cohomological Pieri rules is to use that any
cohomological Pieri coefficient $c^\nu_{\la,p}$ counts the number of
points in an intersection of three Schubert varieties, one of which is
special \cite{hodge:intersection, sottile:pieris,
  bergeron.sottile:pieri-type, buch.kresch.ea:quantum}.  We prove our
formulas by using a $K$-theoretic adaption of this method.  For
Grassmannians of type A, this gives a new and geometric proof of
Lenart's Pieri rule \cite{lenart:combinatorial}.

The $K$-theoretic analogue of the triple intersections
(\ref{eqn:ctrip}) for Pieri coefficients are the numbers
$\euler{X}(\cO^\la \cdot \cO^p \cdot \cO^{\nu^\vee})$, where
$\euler{X} : K(X) \to \Z$ denotes the sheaf Euler characteristic map
(see \S\ref{sec:ktheory}).  However, the Pieri coefficients of
interest are given by $c^\nu_{\la,p} = \euler{X}(\cO^\la \cdot \cO^p
\cdot \cO_\nu^*)$, where $\cO_\nu^* \in K(X)$ is the $K$-theoretic
dual Schubert class of $X^\nu$, a class different from
$\cO^{\nu^\vee}$.  Our proof therefore requires a combinatorial
translation from triple intersection numbers to structure constants.
Furthermore, the $K$-theoretic triple intersection numbers are harder
to compute because they can be non-zero when the intersection of
general translates of $X^\la$, $X^p$, and $X^{\nu^\vee}$ has positive
dimension.  Here we use a construction from
\cite{buch.kresch.ea:quantum} to translate the computation of
intersection numbers to the $K$-theory ring of a projective space.  To
make this construction work in $K$-theory, we also need a Gysin
formula that was proved in \cite{buch.mihalcea:quantum} as an
application of a vanishing theorem of Koll\'ar \cite{kollar:higher}.
To satisfy the conditions of the Gysin formula, we have to prove that
a map from a Richardson variety to projective space has rational
general fibers and that its image has rational singularities.  For
Grassmannians of type A we obtain that the intersection number
$\euler{X}(\cO^\la \cdot \cO^p \cdot \cO^{\nu^\vee})$ is equal to one
whenever the intersection of arbitrary translates of $X^\la$, $X^p$,
and $X^{\nu^\vee}$ is not empty, and otherwise it is zero.  For
maximal orthogonal Grassmannians and Lagrangian Grassmannians, the
integer $\euler{X}(\cO^\la \cdot \cO^p \cdot \cO^{\nu^\vee})$ is equal
to the sheaf Euler characteristic of a complete intersection of linear
and quadric hypersurfaces in projective space.  We note that the
earlier published proofs of $K$-theoretic Pieri rules in
\cite{lenart:combinatorial, lenart.sottile:pieri-type,
  lenart.postnikov:affine} are combinatorial and do not rely on triple
intersections.  On the other hand, the method used here is likely to
work for other homogeneous spaces $G/P$.

Our paper is organized as follows.  In section \ref{sec:ktheory} we
recall some definitions and results related to the $K$-theory of
varieties.  Grassmannians of type A are then handled in section
\ref{sec:typeA}, maximal orthogonal Grassmannians in section
\ref{sec:typeB}, and Lagrangian Grassmannians in section
\ref{sec:typeC}.

We thank Itai Feigenbaum and Emily Sergel for their role in connecting
our Pieri formula for maximal orthogonal Grassmannians to the
conjecture of Thomas and Yong.

\section{The Grothendieck ring}\label{sec:ktheory}

In this section we recall some facts about the $K$-theory of algebraic
varieties; more details and references can be found in
\cite{fulton:intersection}.  The $K$-homology group $K_\circ(X)$ of an
algebraic variety $X$ is the Grothendieck group of coherent
$\cO_X$-modules, i.e.\ the free abelian group generated by isomorphism
classes $[\cF]$ of coherent sheaves on $X$, modulo the relations
$[\cF] = [\cF'] + [\cF'']$ whenever there exists a short exact
sequence $0 \to \cF' \to \cF \to \cF'' \to 0$.  This group is a module
over the $K$-cohomology ring $K^\circ(X)$, defined as the Grothendieck
group of algebraic vector bundles on $X$.  Both the multiplicative
structure of $K^\circ(X)$ and the module structure of $K_\circ(X)$ is
defined by tensor products.  Any closed subvariety $Z \subset X$ has a
{\em Grothendieck class\/} $[\cO_Z] \in K_\circ(X)$ defined by its
structure sheaf.  If $X$ is non-singular, then the implicit map
$K^\circ(X) \to K_\circ(X)$ that sends a vector bundle to its sheaf of
sections is an isomorphism; the inverse map is given by $[\cF] \mapsto
\sum_{i \geq 0} (-1)^i [\cE_i]$ where $0 \to \cE_r \to \dots \to \cE_1
\to \cE_0 \to \cF \to 0$ is any locally free resolution of the
coherent sheaf $\cF$.  In this case we will write simply $K(X)$ for
both $K$-theory groups.

Any morphism of varieties $f : X \to Y$ gives a ring homomorphism $f^*
: K^\circ(Y) \to K^\circ(X)$ defined by pullback of vector bundles.
If $f$ is proper, there is also a pushforward map $f : K_\circ(X) \to
K_\circ(Y)$ defined by $f_*[\cF] = \sum_{i\geq 0} (-1)^i [R^i f_*
\cF]$.  Both pullback and pushforward are functorial with respect to
composition of morphisms.  The projection formula states that
$f_*(f^*(\alpha) \cdot \beta) = \alpha \cdot f_*(\beta)$ for all
$\alpha \in K^\circ(Y)$ and $\beta \in K_\circ(X)$.  If $X$ is a
complete variety, then the sheaf Euler characteristic map $\euler{X} :
K_\circ(X) \to K_\circ(\pt) = \Z$ is defined as the pushforward along
the structure morphism $X \to \{\pt\}$, i.e.\ $\euler{X}([\cF]) =
\sum_{i \geq 0} (-1)^i \dim H^i(X,\cF)$.  If $X$ is irreducible and
rational with rational singularities then $\euler{X}([\cO_X]) = 1$
\cite[p.~494]{griffiths.harris:principles}.

We need the following pushforward formula, which was proved in
\cite[Thm.~3.1]{buch.mihalcea:quantum} as an application of a
vanishing theorem of Koll\'ar \cite[Thm.~7.1]{kollar:higher}.

\begin{lemma}\label{lem:gysin}
  Let $f : X \to Y$ be a surjective map of projective varieties with
  rational singularities.  Assume that $f^{-1}(y)$ is an irreducible
  rational variety for all closed points in a dense open subset of
  $Y$.  Then $f_*[\cO_X] = [\cO_Y ] \in K_\circ(Y)$.
\end{lemma}

We also need the following well known fact.

\begin{lemma}
  Let $X$ be a non-singular variety and let $Y$ and $Z$ be closed
  varieties of $X$ with Cohen-Macaulay singularities.  Assume that
  each component of $Y \cap Z$ has dimension
  $\dim(Y)+\dim(Z)-\dim(X)$.  Then $Y\cap Z$ is Cohen-Macaulay and
  $[\cO_Y] \cdot [\cO_Z] = [\cO_{Y\cap Z}]$ in $K(X)$.
\end{lemma}
\begin{proof}
  The diagonal embedding $\Delta : X \to X \times X$ is a regular
  embedding \cite[B.7.3]{fulton:intersection}, and any local regular
  sequence defining the ideal of $X$ in $X\times X$ restricts to a
  local regular sequence defining the ideal of $Y \cap Z$ in $Y \times
  Z$ by \cite[A.7.1]{fulton:intersection}.  This implies that
  $\cTor_i^{X\times X}(\cO_X, \cO_{Y\times Z}) = 0$ for all $i>0$, so
  $[\cO_Y] \cdot [\cO_Z] = \Delta^*[\cO_{Y \times Z}] = \sum_{i \geq
    0} (-1)^i [\cTor_i^{X\times X}(\cO_X, \cO_{Y\times Z})] = [\cO_Y
  \otimes \cO_Z] = [\cO_{Y\cap Z}]$, as required.
\end{proof}

Now let $X = G/P$ be a homogeneous space defined by a complex
connected semisimple linear algebraic group $G$ and a parabolic
subgroup $P$; all Grassmannians can be constructed in this fashion.
The Schubert varieties of $X$ relative to a Borel subgroup $B \subset
G$ are the closures of the $B$-orbits in $X$, and the Grothendieck
classes of these varieties form a $\Z$-basis for the $K$-theory ring
$K(X)$.  It is known that all Schubert varieties are Cohen-Macaulay
and have rational singularities \cite{mehta.srinivas:normality,
  lauritzen.thomsen:frobenius}.  As a special case this implies that
an quadric hypersurfaces in $\C^n$ has rational singularities, since
it is isomorphism to a product of a smaller affine space with an open
subset of the Schubert divisor in an orthogonal Grassmannian
$\OG(1,m)$ (see e.g.\ \cite[\S 4]{buch.kresch.ea:quantum}).

A Richardson variety is any non-empty intersection $Y \cap Z$,
where $Y$ is a Schubert variety relative to $B$ and $Z$ is a Schubert
variety relative to the opposite Borel subgroup $B^\op$.  It was
proved by Deodhar that any Richardson variety is irreducible and
rational \cite{deodhar:on}, and Brion has proved in that Richardson
varieties have rational singularities \cite{brion:positivity}.  In
particular, the sheaf Euler characteristic of a Richardson variety is
equal to one.

Finally we mention that a Schubert variety in the projective space
$\P^n$ is the same as a linear subspace.  The Grothendieck ring is
given by $K(\P^n) = \Z[t]/(t^{n+1})$, where $t$ is the class of a
hyperplane.  The class of a quadric hypersurface is equal to $2t-t^2$,
and a linear subspace of codimension $i$ has class $t^i$.  In
particular, the map $\euler{\P^n} : K(\P^n) \to \Z$ is determined by
$\euler{\P^n}(t^i) = 1$ for $0 \leq i \leq n$.

\section{Grassmannians of type A}\label{sec:typeA}

Let $X = \Gr(m,n) = \{ V \subset \C^n \mid \dim(V)=m \}$ be the
Grassmann variety of $m$-planes in $\C^n$.  This is a non-singular
variety of complex dimension $mk$, where $k = n-m$.  The Schubert
varieties in $X$ are indexed by partitions $\la = (\la_1 \geq \dots
\geq \la_m \geq 0)$ such that $\la_1 \leq k$.  Equivalently, the Young
diagram of $\la$ can be contained in a rectangle $m$ rows and $k$
columns.  We will identify $\la$ with its Young diagram.  The Schubert
variety for $\la$ relative to a complete flag $0 \subsetneq F_1
\subsetneq F_2 \subsetneq \dots \subsetneq F_n = \C^n$ is defined by
\begin{equation*}
 X^\la(F_\bull) = \{V \in X \mid
   \dim(V\cap F_{k+i-\la_i}) \geq i ~\forall 1 \leq i \leq m \} \,.
\end{equation*}
The codimension of this variety in $X$ is equal to the weight $|\la| =
\sum \la_i$.  If $u_1,\dots,u_r$ are vectors in a complex vector
space, then we let $\bspan{u_1,\dots,u_r}$ denote their span.  Let
$e_1,\dots,e_n$ be the standard basis of $\C^n$.  We will mostly
consider Schubert varieties relative to the standard flags in $\C^n$,
defined by $E_i = \bspan{e_1,\dots,e_i}$ and $E^\op_i =
\bspan{e_{n+1-i},\dots,e_n}$.  Let $\cO^\la = [\cO_{X^\la}] \in K(X)$
denote the class of $X^\la := X^\la(E_\bull)$.  The $K$-theoretic
Schubert structure constants $c^\nu_{\la\mu}$ for $X$ are defined by
equation (\ref{eqn:kcoef}), where the sum includes all partitions
$\nu$ contained in the $m\times k$-rectangle.  The goal of this
section is to give a simple geometric proof of Lenart's Pieri rule for
the special coefficients $c^\nu_{\la,p}$, where we identify the
integer $p \in \N$ with the one-part partition $(p)$.

Given a partition $\mu$ contained in the $m\times k$-rectangle, let
$\mu^\vee = (k-\mu_m,\dots,k-\mu_1)$ denote the Poincare dual
partition.  We will also call this partition the $m\times k$--dual of
$\mu$.  The intersection $X^\la(E_\bull) \cap X^\mu(E^\op_\bull)$ is
non-empty if and only if $\la \subset \mu^\vee$.  Assume that $\la
\subset \mu^\vee$ and let $\theta = \mu^\vee / \la$ be the skew
diagram of boxes in $\mu^\vee$ that are not in $\la$.  This is the set
of boxes remaining when the boxes of $\la$ are deleted from the
upper-left corner of the $m \times k$-rectangle, and the boxes of
$\mu$ (rotated by 180 degrees) are deleted from the lower-right
corner.
\begin{equation*}
\pic{.5}{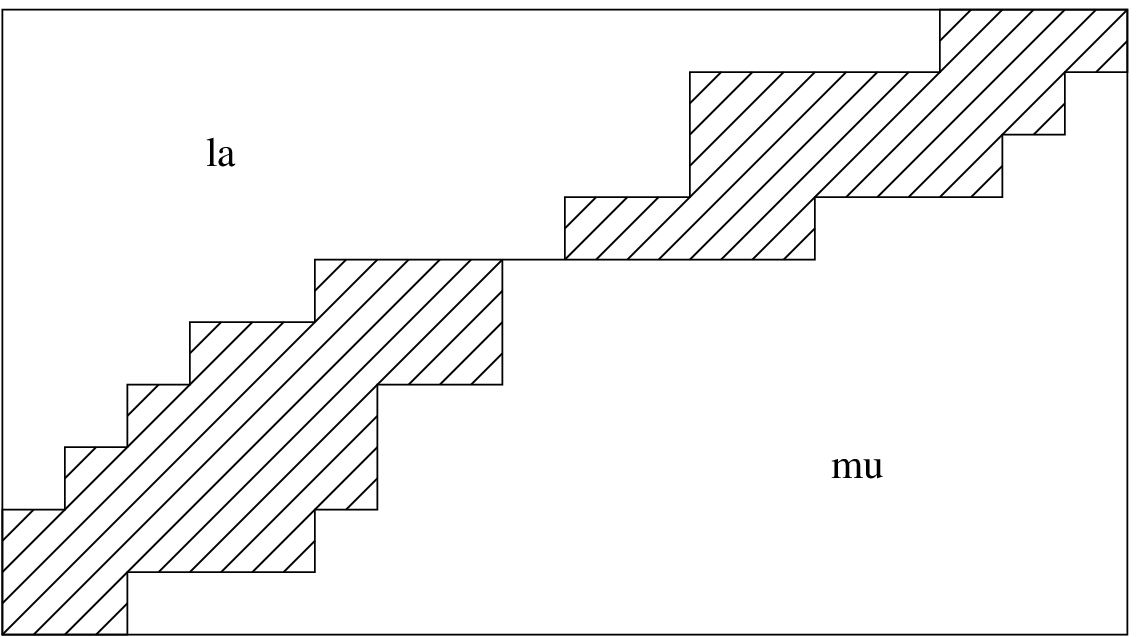}
\end{equation*}
Define the Richardson variety $X_\theta = X^\la(E_\bull) \cap
X^\mu(E^\op_\bull) \subset X$.  This variety has dimension $|\theta| =
mk-|\la|-|\mu|$.  As a special case, notice that $X_\la =
X^{\la^\vee}(E^\op_\bull)$ is the dual Schubert variety of $X^\la$.
Let $\cO_\theta = [\cO_{X_\theta}] = \cO^\la \cdot \cO^\mu \in K(X)$
denote the Grothendieck class of $X_\theta$.  While $X_\theta$ depends
on the partitions $\la$ and $\mu$, it follows from
Lemma~\ref{lem:chopA} below that its isomorphism class depends only on
the skew diagram $\theta$.  For any vector $u \in \C^n$ we define
$X_\theta(u) = \{ V \in X_\theta \mid u \in V \}$.  Let $\bigcup
X_\theta = \bigcup_{V \in X_\theta} V \subset \C^n$ be the set of
vectors $u \in \C^n$ for which $X_\theta(u) \neq \emptyset$.  This set
is an irreducible closed subvariety of $\C^n$, because $\bigcup
X_\theta = \pi_2(\pi_1^{-1}(X_\theta))$ where $\pi_1 : \cS \to X$ and
$\pi_2 : \cS \to \C^n$ are the natural projections from the
tautological subbundle $\cS = \{(V,u) \in X \times \C^n \mid u \in V
\}$.  We wish to show that if $u$ is a general vector in $\bigcup
X_\theta$, then $X_\theta(u)$ is a Richardson variety in the
Grassmannian $\Gr(m-1,\C^n/\bspan{u})$, which we identify with the
$m$-planes in $X$ that contain $u$.  In particular, $X_\theta(u)$ is
irreducible and rational.  The following example shows that this may
be false without the assumption that $u$ is general.

\begin{example}
  Let $X = \Gr(3,5)$ and $\la = \mu = (1)$.  Then $\bigcup X_\theta =
  \C^5$.  Set $u = e_1+e_5$.  Then $X_\theta(u)$ has two components,
  which can be naturally identified with $\P(\C^5/\bspan{e_1,e_5})$
  and $\Gr(2,\bspan{e_1,e_2,e_4,e_5}/\bspan{e_1+e_5})$.  These
  components meet in the line $\P(\bspan{e_2,e_4})$.
\end{example}

Assume that $a \in [0,m]$ and $b \in [0,k]$ are integers such that
$\la_a \geq b$ and $\mu_{m-a} \geq k-b$ (here we set $\la_0 = \mu_0 =
k$).  This implies that the diagram $\theta = \mu^\vee/\la$ can be
split into a north-east part $\theta'$ in rows 1 through $a$ and a
south-west part $\theta''$ in rows $a+1$ through $m$.
\begin{equation*} 
\pic{.5}{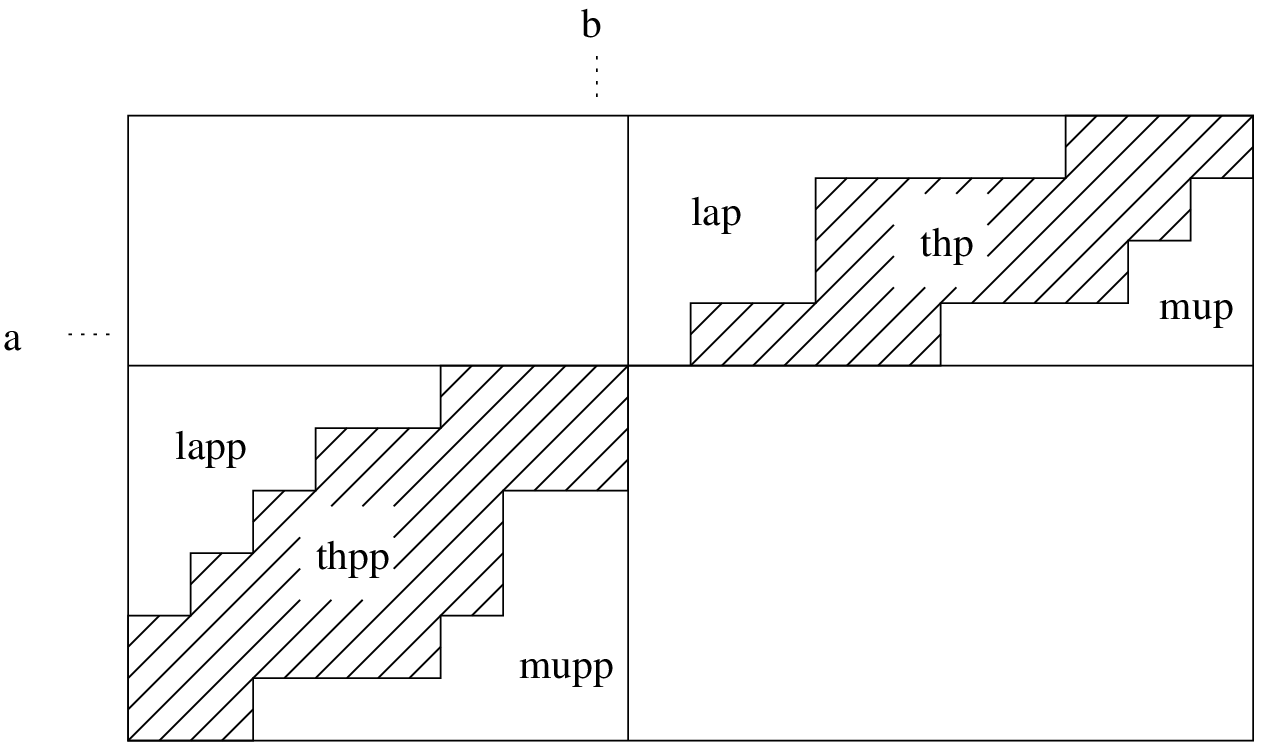}
\end{equation*} 
Set $\la' = (\la_1-b,\dots,\la_a-b)$, $\mu' =
(\mu_{m-a+1},\dots,\mu_m)$, and $\theta' = {\mu'}^\vee/\la'$, where
${\mu'}^\vee$ is the $a \times (k-b)$--dual of $\mu'$.  This skew
diagram defines a Richardson variety $X'_{\theta'}$ in the Grassmannian
$X' = \Gr(a,E_{k+a-b})$, where we use the (ordered) basis
$e_1,\dots,e_{k+a-b}$ for $E_{k+a-b}$.  Similarly we set $\la'' =
(\la_{a+1},\dots,\la_m)$, $\mu'' = (\mu_1-k+b,\dots,\mu_{m-a}-k+b)$,
and $\theta'' = {\mu''}^\vee/\la''$, which defines the Richardson
variety $X''_{\theta''}$ in $X'' = \Gr(m-a,E^\op_{m+b-a})$, using the
basis $e_{k+a-b+1},\dots,e_n$ for $E^\op_{m+b-a}$.  Set $\bigcup
X'_{\theta'} = \bigcup_{V' \in X'_{\theta'}} V' \subset E_{k+a-b}$ and
$\bigcup X''_{\theta''} = \bigcup_{V''\in X''_{\theta''}} V'' \subset
E^\op_{m+b-a}$.

\begin{lemma}\label{lem:chopA}
  {\rm(a)} We have $\bigcup X_\theta = \bigcup X'_{\theta'} \times
  \bigcup X''_{\theta''}$ in $\C^n = E_{k+a-b} \times
  E^\op_{m+b-a}$.\smallskip

  \noindent
  {\rm(b)} For arbitrary vectors $u' \in E_{k+a-b}$ and $u'' \in
  E_{m+b-a}$, the inclusion $X' \times X'' \subset X$ defined by
  $(V',V'') \mapsto V'\oplus V''$ identifies $X'_{\theta'}(u') \times
  X''_{\theta''}(u'')$ with $X_\theta(u' + u'')$.
%
\end{lemma}
\begin{proof}
  If $V \in X_\theta$ then $\dim(V \cap E_{k+a-b}) \geq \dim(V \cap
  E_{k+a-\la_a}) \geq a$ and $\dim(V \cap E^\op_{m+b-a}) \geq \dim(V
  \cap E^\op_{k+(m-a)-\mu_{m-a}}) \geq m-a$.  It follows that $V = V'
  \oplus V''$ where $V' = V \cap E_{k+a-b} \in X'$ and $V'' = V \cap
  E^\op_{m+b-a} \in X''$.  Given arbitrary points $V' \in X'$ and $V''
  \in X''$ it is immediate from the definitions that $V'\oplus V'' \in
  X_\theta(u'+u'')$ if and only if $V' \in X'_{\theta'}(u')$ and $V''
  \in X''_{\theta''}(u'')$.  The lemma follows from this.
\end{proof}

Let $\wb\theta$ be the diagram obtained from $\theta$ by removing the
bottom box in each non-empty column.  Let $c(\theta) = |\theta| -
|\wb\theta|$ be the number of non-empty columns.

\begin{lemma}\label{lem:fiberA}
  {\rm(a)} The set $\bigcup X_\theta$ is a linear subspace of $\C^n$
  of dimension $m + c(\theta)$.\smallskip

  \noindent
  {\rm(b)} For all vectors $u$ in a dense open subset of $\bigcup
  X_\theta$ we have $X_\theta(u) \cong X_{\wb\theta}$.
\end{lemma}
\begin{proof}
  If $\theta = \emptyset$, then $X_\theta$ is a single point and the
  lemma is clear.  Using Lemma~\ref{lem:chopA} it is therefore enough
  to assume that $\theta$ has only one component which contains the
  upper-right and lower-left boxes of the $m\times k$-rectangle.  This
  implies that $\la_i+\mu_{m-i} \leq k$ for $0 \leq i \leq m$ (recall
  that we set $\la_0=\mu_0=k$).  Let $u \in \C^n$ be any vector such
  that all coordinates of $u$ are non-zero.  It is enough to show that
  $X_\theta(u) \cong X_{\wb\theta}$.  Set $\wb E = \C^n/\bspan{u}$ and
  define flags in this vector space by $\wb E_i =
  (E_i+\bspan{u})/\bspan{u}$ and $\wb E^\op_i =
  (E^\op_i+\bspan{u})/\bspan{u}$, for $1 \leq i \leq n-1$.  Since we
  have $u \not\in E_{i}+E^\op_{n-1-i}$ for each $i$, it follows that
  $\wb E_i \cap \wb E^\op_{n-1-i} = 0$, so the flags $\wb E_\bull$ and
  $\wb E^\op_\bull$ are opposite.  Identify $\wb X = \Gr(m-1,\wb E)$
  with the set of $m$-planes $V \in X$ for which $u \in V$.  Then we
  have $X^\la(E_\bull) \cap \wb X = \wb X^\la(\wb E_\bull)$ and
  $X^\mu(E^\op_\bull) \cap \wb X = \wb X^\mu(\wb E^\op_\bull)$, so
  $X_\theta(u) = X_\theta \cap \wb X = \wb X^\la(\wb E_\bull) \cap \wb
  X^\mu(\wb E^\op_\bull) = \wb X_{\wb\theta}$, as required.
\end{proof}

We can now prove our formula for the $K$-theoretic triple intersection
numbers $\euler{X}(\cO^\la \cdot \cO^p \cdot \cO^\mu)$.  They turn out
to be simpler than the corresponding Pieri coefficients
$c^\nu_{\la,p}$ (see eqn.\ (\ref{eqn:pieriA}) below), despite the fact
that the more general intersection numbers $\euler{X}(\cO^\la \cdot
\cO^\nu \cdot \cO^\mu)$ are not well behaved \cite[\S
8]{buch:littlewood-richardson}.

\begin{prop}\label{prop:tripleA}
  Let $\theta$ be a skew diagram contained in the $m\times
  k$-rectangle and let $0 \leq p \leq k$.  Then we have
  \begin{equation*} 
  \euler{X}(\cO_\theta \cdot \cO^p) =
  \begin{cases} 1 & \text{if $p \leq c(\theta)$;} \\
    0 & \text{if $p > c(\theta)$.}
  \end{cases}
  \end{equation*}
\end{prop}
\begin{proof}
  Let $Z = \Fl(1,m;n)$ be the variety of all two-step flags $L \subset
  V \subset \C^n$ with $\dim(L)=1$ and $\dim(V)=m$, and let $\pi_1 : Z
  \to \P^{n-1}$ and $\pi_m : Z \to X$ be the projections.  Since
  $\pi_m : \pi_1^{-1}(\P(E_{k+1-p})) \to X^p$ is a birational
  isomorphism of varieties with rational singularities and $\pi_1$ is
  flat, it follows that ${\pi_m}_* \pi_1^* [\cO_{\P(E_{k+1-p})}] =
  {\pi_m}_* [\cO_{\pi_1^{-1}(\P(E_{k+1-p}))}] = \cO^p \in K(X)$.
  Lemma~\ref{lem:fiberA} implies that the general fibers of the map
  $\pi_1 : \pi_m^{-1}(X_\theta) \to \pi_1(\pi_m^{-1}(X_\theta))$ are
  rational.  Since the Richardson variety $\pi_m^{-1}(X_\theta)$ has
  rational singularities, we therefore deduce from
  Lemma~\ref{lem:gysin} that ${\pi_1}_* \pi_m^* [\cO_{X_\theta}] =
  {\pi_1}_* [\cO_{\pi_m^{-1}(X_\theta)}] =
  [\cO_{\pi_1(\pi_m^{-1}(X_\theta))}] = [\cO_{\P(\bigcup X_\theta)}]
  \in K(\P^{n-1})$.  Using the projection formula we obtain
  \begin{equation*}\begin{split}
    \euler{X}(\cO_\theta \cdot \cO^p) 
    &= \euler{X}([\cO_{X_\theta}]\cdot {\pi_m}_* \pi_1^* 
       [\cO_{\P(E_{k+1-p})}]) \\
    &= \euler{\P^{n-1}}({\pi_1}_* \pi_m^*[\cO_{X_\theta}] \cdot
       [\cO_{\P(E_{k+1-p})}]) \\
    &= \euler{\P^{n-1}}([\cO_{\P(\bigcup X_\theta)}] \cdot
       [\cO_{\P(E_{k+1-p})}]) \,.
  \end{split}
  \end{equation*}
  Lemma~\ref{lem:fiberA} implies that this Euler characteristic equals
  one if $m+c(\theta) + k+1-p > n$ and is zero otherwise, as required.
\end{proof}

In order use the intersection numbers of
Proposition~\ref{prop:tripleA} to compute the Pieri coefficients
$c^\nu_{\la,r}$, we need the dual Schubert classes in $K(X)$.  Recall
that a skew diagram is a {\em horizontal strip\/} if it contains at
most one box in each column, and a {\em vertical strip\/} if it has at
most one box in each row.  The diagram is a {\em rook strip\/} if it
is both a horizontal strip and a vertical strip.  For any partition
$\nu$ in the $m\times k$-rectangle we define
\begin{equation}\label{eqn:dualA}
  \cO_\nu^* = \sum_{\nu/\tau\text{ rook strip}} 
  (-1)^{|\nu/\tau|}\, \cO_\tau
\end{equation}
where the sum is over all partitions $\tau \subset \nu$ such that
$\nu/\tau$ is a rook strip.  The following lemma implies that
$c^\nu_{\la\mu} = \euler{X}(\cO^\la \cdot \cO^\mu \cdot \cO_\nu^*)$.
In particular, we have $c^\nu_{\la\mu} = 0$ whenever $\la \not\subset
\nu$.  A different proof of the lemma can be found in \cite[\S
8]{buch:littlewood-richardson}.

\begin{lemma}\label{lem:dualA}
  Let $\mu$ and $\nu$ be partitions contained in the $m \times
  k$-rectangle.  Then we have $\euler{X}(\cO_\nu^* \cdot \cO^\mu) =
  \delta_{\nu,\mu}$ (Kronecker's delta).
\end{lemma}
\begin{proof}
  Since a non-empty Richardson variety is irreducible, rational, and
  has rational singularities, it follows that $\euler{X}(\cO_\tau
  \cdot \cO^\mu)$ is equal to one if $\mu \subset \tau$ and is zero
  otherwise.  This shows that $\euler{X}(\cO_\nu^* \cdot \cO^\mu) = 0$
  when $\mu \not\subset \nu$ and $\euler{X}(\cO^*_\nu \cdot \cO^\nu) =
  \euler{X}(\cO_\nu \cdot \cO^\nu) = 1$.  Assume that $\mu \subsetneq
  \nu$ and let $\la$ be the smallest partition such that $\mu \subset
  \la \subset \nu$ and $\nu/\la$ is a rook strip.  Then
  $\euler{X}(\cO^*_\nu \cdot \cO^\mu) = \sum_{\la \subset \tau \subset
    \nu} (-1)^{|\nu/\tau|}$ is a sum of $2^{|\nu/\la|}$ terms, half of
  which are negative.  The lemma follows from this.
\end{proof}

A {\em south-east corner\/} of the skew diagram $\theta$ is any box $B
\in \theta$ such that $\theta$ does not contain a box directly below
or directly to the right of $B$.  Let $\theta'$ be the diagram
obtained from $\theta$ by removing its south-east corners.  Notice
that $\theta$ is a rook strip if and only if $\theta' = \emptyset$.
Given any integer $p \in \Z$, we will abuse notation and write
\begin{equation*}
  \euler{X}(\cO_\theta \cdot \cO^p) = \begin{cases}
    1 & \text{if $p \leq c(\theta)$;} \\
    0 & \text{if $p > c(\theta)$.}
  \end{cases}
\end{equation*}
This is equivalent to setting $\cO^p = 1$ for $p < 0$ and working on a
sufficiently large Grassmannian $X$.  Using this convention we define
the constants
\begin{equation}\label{eqn:sumA}
  \cA(\theta,p) = \sum_{\theta' \subset \varphi \subset \theta}
  (-1)^{|\theta|-|\varphi|}\, \euler{X}(\cO_\varphi \cdot \cO^p)
\end{equation}
where the sum is over all skew diagrams $\varphi$ obtained by removing
a subset of the south-east corners from $\theta$.
Proposition~\ref{prop:tripleA} and Lemma~\ref{lem:dualA} imply the
following.

\begin{cor}\label{cor:cisA}
  Let $\la \subset \nu$ be partitions contained in the $m \times
  k$-rectangle and let $0 \leq p \leq k$.  Then $c^\nu_{\la,p} =
  \cA(\nu/\la,p)$.
\end{cor}

Lenart's Pieri rule \cite{lenart:combinatorial} states that
$c^\nu_{\la,p}$ is non-zero only if $\la \subset \nu$ and $\nu/\la$ is
a horizontal strip, in which case we have
\begin{equation}\label{eqn:pieriA}
  c^\nu_{\la,p} = (-1)^{|\nu/\la|-p} \binom{r(\nu/\la)-1}{|\nu/\la|-p} \,.
\end{equation}
Here $r(\nu/\la)$ denotes the number of non-empty rows of the skew
diagram $\nu/\la$.  To be consistent with our treatment of Pieri
coefficients for Lagrangian and orthogonal Grassmannians, we will
restate this rule as a set of recursive identities among the integers
$\cA(\theta,p)$.  Given a horizontal strip $\theta$, we let
$\wh\theta$ be the diagram obtained by removing the top row of
$\theta$.  Notice that $\theta$ is a single row of boxes if and only
if $\wh\theta = \emptyset$.  Lenart's formula is equivalent to the
following.

\begin{thm}
  Let $\theta$ be a skew diagram and let $p \in \Z$.  If $\theta$ is
  not a horizontal strip then $\cA(\theta,p)=0$.  If $p \leq 0$ then
  $\cA(\theta,p)=\delta_{\theta,\emptyset}$, and $\cA(\emptyset,p)$ is
  equal to one if $p\leq 0$ and is zero otherwise.  If $\theta \neq
  \emptyset$ is a horizontal strip and $p > 0$, then $\cA(\theta,p)$
  is determined by the following rules.\smallskip
  
  \noindent{\rm(i)}
  If $\wh\theta = \emptyset$, then $\cA(\theta,p) =
  \delta_{|\theta|,p}$.\smallskip

  \noindent{\rm(ii)}
  If $\wh\theta \neq \emptyset$, then $\cA(\theta,p) =
  \cA(\wh\theta,p-a) - \cA(\wh\theta, p-a+1)$, where $a =
  |\theta|-|\wh\theta|$.
\end{thm}
\begin{proof}
  We may assume that $\theta \neq \emptyset$ and $p>0$, since
  otherwise the result follows from Corollary~\ref{cor:cisA}.  If
  $\theta$ is not a horizontal strip, then we can find a box $B \in
  \theta \ssm \theta'$ such that the box directly above $B$ is
  contained in $\theta'$; if $\varphi$ is any skew diagram such that
  $\theta' \subset \varphi \subset \theta$ and $B \notin \varphi$,
  then $c(\varphi) = c(\varphi \cup B)$ and $\euler{X}(\cO_\varphi
  \cdot \cO^p) = \euler{X}(\cO_{\varphi \cup B} \cdot \cO^p)$, so the
  terms of (\ref{eqn:sumA}) given by $\varphi$ and $\varphi \cup B$
  cancel each other out.  And if $\theta$ is a single row, then
  $\cA(\theta,p) = \euler{X}(\cO_\theta\cdot\cO^p) -
  \euler{X}(\cO_{\theta'}\cdot\cO^p) = \delta_{|\theta|,p}$.

  Assume that $\theta$ is a horizontal strip with two or more rows.
  Let $\psi = \theta \ssm \wh\theta$ be the top row and set $\psi' =
  \psi \cap \theta'$ and $\wh\theta' = \wh\theta \cap \theta'$.  Each
  term of the sum (\ref{eqn:sumA}) is given by a skew diagram of the
  form $\varphi \cup \psi$ or $\varphi \cup \psi'$, where $\wh\theta'
  \subset \varphi \subset \wh\theta$.  Since $c(\varphi) = c(\varphi
  \cup \psi) - a = c(\varphi \cup \psi') - a + 1$, it follows from
  Proposition~\ref{prop:tripleA} that
  \begin{equation*}
    \begin{split}
      & (-1)^{|\theta|-|\varphi \cup \psi|}
      \left(\euler{X}(\cO_{\varphi \cup \psi} \cdot \cO^p)
      - \euler{X}(\cO_{\varphi \cup \psi'} \cdot \cO^p)\right) \\
      &=
      (-1)^{|\wh\theta|-|\varphi|}
      \left(\euler{X}(\cO_\varphi \cdot \cO^{p-a})
      - \euler{X}(\cO_\varphi \cdot \cO^{p-a+1})\right) .
    \end{split}
  \end{equation*}
  This implies that $\cA(\theta,p) = \cA(\wh\theta,p-a) -
  \cA(\wh\theta,p-a+1)$ by summing over $\varphi$.
\end{proof}

\section{Maximal orthogonal Grassmannians}\label{sec:typeB}

Let $e_1,\dots,e_{2n+1}$ be the standard basis for $\C^{2n+1}$.
Define an orthogonal form on $\C^{2n}$ by $(e_i,e_j) =
\delta_{i+j,2n+2}$.  A vector subspace $U \subset \C^{2n+1}$ is called
{\em isotropic\/} if $(U,U) = 0$.  This implies that $\dim(U) \leq n$.
Let $X = \OG(n,2n+1) = \{ V \subset \C^{2n+1} \mid \dim(V)=n \text{
  and } (V,V)=0 \}$ be the orthogonal Grassmannian of maximal
isotropic subspaces of $\C^{2n+1}$.  This is a non-singular variety of
dimension $\binom{n+1}{2}$.  It is isomorphic to the even orthogonal
Grassmannian $\OG(n+1,2n+2)$, so this variety is also covered in what
follows.  The Schubert varieties in $X$ are indexed by strict
partitions $\la = (\la_1 > \la_2 > \dots > \la_\ell > 0)$ for which
$\la_1 \leq n$.  The length of $\la$ is the number $\ell(\la) = \ell$
of non-zero parts.  The Schubert variety for $\la$ relative to an
isotropic flag $0 \subsetneq F_1 \subsetneq F_2 \subsetneq \dots
\subsetneq F_n \subset \C^{2n+1}$ with $(F_n,F_n)=0$ is defined by
\begin{equation*} 
X^\la(F_\bull) = \{ V \in X \mid 
\dim(V \cap F_{n+1-\la_i}) \geq i ~\forall 1 \leq i \leq \ell(\la) \}\,.
\end{equation*}
The standard isotropic flags of $\C^{2n+1}$ are defined by $E_i =
\bspan{e_1,\dots,e_i}$ and $E^\op_i =
\bspan{e_{2n+2-i},\dots,e_{2n+1}}$.  Set $X^\la = X^\la(E_\bull)$.
The classes $\cO^\la = [\cO_{X^\la}]$ form a $\Z$-basis for the
Grothendieck ring $K(X)$, and this ring is generated as a $\Z$-algebra
by the special classes $\cO^1,\dots,\cO^n$.

When we are working in the context of the maximal orthogonal
Grassmannian $X = \OG(n,2n+1)$, we will identify a strict partition
$\la = (\la_1 > \dots > \la_\ell > 0)$ with its {\em shifted Young
  diagram}.  The $i$-th row in this diagram contains $\la_i$ boxes,
which are preceded by $i-1$ unused positions.  The boxes of the
staircase partition $\rho_n = (n,n-1,\dots,1)$ thus correspond to the
upper-triangular positions in an $n\times n$ matrix.  If $\mu$ is a
strict partition with $\mu_1 \leq n$, then the $n$--dual partition
$\mu^\vee$ consists of the parts of $\rho_n$ which are not parts of
$\mu$.  We have $X^\la(E_\bull) \cap X^\mu(E^\op_\bull) \neq
\emptyset$ if and only if $\la \subset \mu^\vee$.  In this case the
shifted skew diagram $\theta = \mu^\vee/\la$ is obtained by removing
the boxes of $\la$ from the upper-left corner of $\rho_n$ and the
boxes of $\mu$ from the lower-right corner, after mirroring $\mu$ in
the south-west to north-east antidiagonal.  For example, when $n=12$,
$\la = (11,9,8,5,2)$, and $\mu = (10,8,7,4)$, we obtain the following
diagram $\theta$.
\begin{equation*}
\pic{.5}{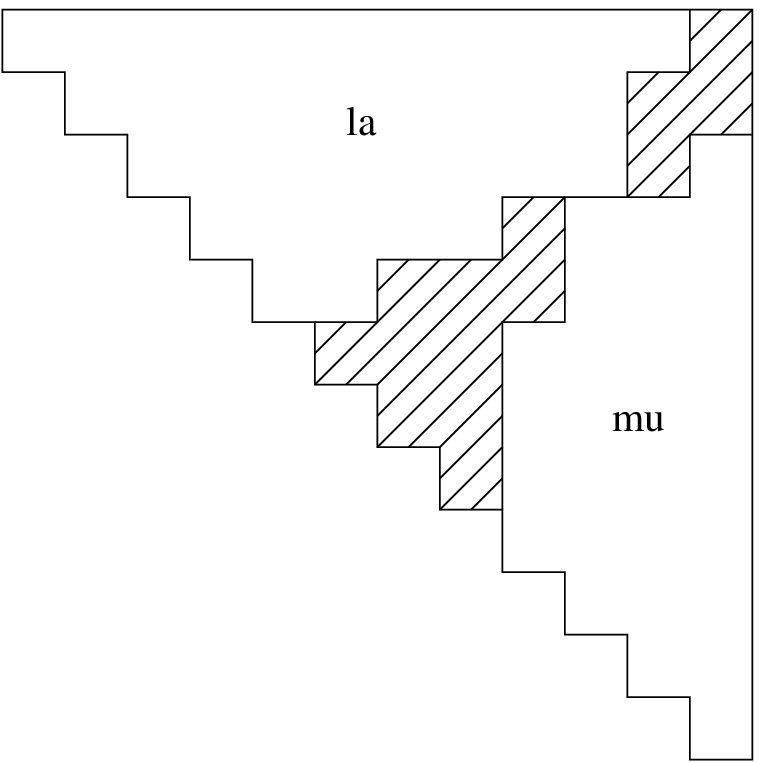}
\end{equation*}
The set of leftmost boxes of $\rho_n$ are called {\em diagonal boxes}.
The above skew diagram $\theta$ contains three such boxes.

Define the Richardson variety $X_\theta = X^\la(E_\bull) \cap
X^\mu(E^\op_\bull) \subset X$.  This variety has dimension $|\theta| =
\binom{n+1}{2} - |\la| - |\mu|$.  It follows from
Lemma~\ref{lem:chopB} below that the isomorphism class of $X_\theta$
depends only on the shape of the skew diagram $\theta$, at least if it
is remembered which boxes of $\theta$ are diagonal boxes.  We set
$\cO_\theta = [\cO_{X_\theta}] = \cO^\la \cdot \cO^\mu \in K(X)$.  For
$u \in \C^{2n+1}$ we define $X_\theta(u) = \{ V \in X_\theta \mid u
\in V \}$.  Let $\bigcup X_\theta = \bigcup_{V\in X_\theta} V \subset
\C^{2n+1}$ be the set of vectors $u$ for which $X_\theta(u) \neq
\emptyset$.

As for Grassmannians of type A, we would like to write the Richardson
variety $X_\theta$ as a product, where the factors correspond to the
components of $\theta$.  Assume that $a,b \geq 0$ are integers such
that $0<a+b<n$, $\la_a > n-a-b$, and $\mu_b > n-a-b$ (we set $\la_0 =
\mu_0 = n+1$).  Then $\theta = \mu^\vee/\la$ can be split into a
north-east part and a south-west part.
\begin{equation*}
\pic{.5}{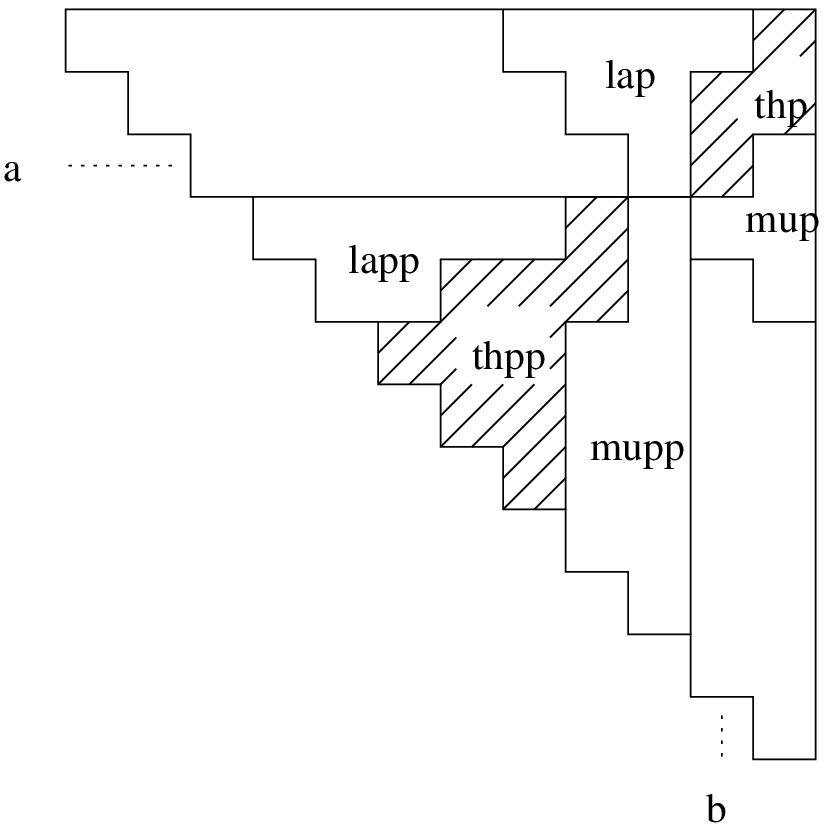}
\end{equation*} 
Set $\la' = (\la_1+a+b-n, \dots, \la_a+a+b-n)$, $\mu' = (\mu_1+a+b-n,
\dots, \mu_b+a+b-n)$, and $\theta' = {\mu'}^\vee/\la'$, where
${\mu'}^\vee$ is the $(a+b)$--dual of $\mu'$.  Set $E' = E_{a+b}
\oplus E^\op_{a+b}$ and extend the orthogonal form on $E'$ to the
vector space $E' \oplus \C$, with basis
$e_1,\dots,e_{a+b},e',e_{2n+2-a-b},\dots,e_{2n+1}$, by setting
$(e',e')=1$ and $(e',e_i)=0$ for every $i$.  Then $\theta'$ defines a
Richardson variety $X'_{\theta'}$ in $X' = \OG(a+b,E'\oplus \C)$.
Similarly we set $\la'' = (\la_{a+1},\dots,\la_{\ell(\la)})$, $\mu'' =
(\mu_{b+1},\dots,\mu_{\ell(\mu)})$, and $\theta'' =
{\mu''}^\vee/\la''$, using the $(n-a-b)$--dual of $\mu''$.  This
defines the Richardson variety $X''_{\theta''}$ in $X'' = \OG(n-a-b,
E'')$, where $E'' = {E'}^\perp = \bspan{e_{a+b+1},\dots,e_{2n+1-a-b}}
\subset \C^{2n+1}$.  Set $\wt X' = \{ V' \in X' \mid V' \subset E'
\}$.  If $V' \in X'_{\theta'}$ then $\dim(V' \cap E_{a+b}) \geq a$ and
$\dim(V' \cap E^\op_{a+b}) \geq b$, which implies that $V' \in \wt
X'$.  In particular, we have $\bigcup X'_{\theta'} \subset E'$.

\begin{lemma}\label{lem:chopB}
  {\rm(a)} We have $\bigcup X_\theta = \bigcup X'_{\theta'} \times
  \bigcup X''_{\theta''}$ in $\C^{2n+1} = E' \times E''$.\smallskip

  \noindent
  {\rm(b)} For arbitrary vectors $u' \in E'$ and $u'' \in E''$, the natural
  inclusion $\wt X' \times X'' \subset X$ defined by $(V',V'') \mapsto
  V'\oplus V''$ identifies $X'_{\theta'}(u') \times
  X''_{\theta''}(u'')$ with $X_\theta(u'+u'')$.
\end{lemma}
\begin{proof}
  If $V \in X_\theta$ then $\dim(V \cap E_{a+b}) \geq \dim(V \cap
  E_{n+1-\la_a}) \geq a$ and $\dim(V \cap E^\op_{a+b}) \geq \dim(V
  \cap E^\op_{n+1-\la_b}) \geq b$.  This implies that $V' = V \cap E'$
  is a maximal isotropic subspace in $E' \oplus \C$, so $V' \in \wt
  X'$.  It also follows that $\dim(V \cap {E'}^\perp) = n-a-b$, so
  $V'' = V \cap E'' \in X''$.  Given arbitrary points $V' \in X'$ and
  $V'' \in X''$, it follows from the definitions that $V'\oplus V''
  \in X_\theta(u'+u'')$ if and only $V' \in X'_{\theta'}(u')$ and $V''
  \in X''_{\theta''}(u'')$.  The lemma follows from this.
\end{proof}

The {\em south-east rim\/} of the shifted skew diagram $\theta$ is the
set of boxes $B \in \theta$ such that no box of $\theta$ is located
strictly south and strictly east of $B$.  Let $\wb\theta$ denote the
diagram obtained by removing the south-east rim from $\theta$.  We say
that $\theta$ is a {\em rim\/} if $\wb\theta = \emptyset$.  Let
$d(\theta) = |\theta| - |\wb\theta|$ be number of boxes in the
south-east rim.  Let $N(\theta)$ denote the number of connected
components of the diagram $\theta$, where two boxes are connected if
they share a side.  Set $N^-(\theta) = \max(N(\theta)-1,0)$.  For use
with the Lagrangian Grassmannian, we also let $N'(\theta)$ be the
number of components that do not contain any diagonal boxes.  The
diagram displayed above gives $d(\theta) = 10$, $N(\theta)=2$, and
$N^-(\theta) = N'(\theta) = 1$.

\begin{lemma}\label{lem:fiberB}
  {\rm(a)} The set $\bigcup X_\theta \subset \C^{2n+1}$ is a scheme
  theoretic complete intersection with rational singularities.  It has
  dimension $n+d(\theta)$ and is defined by $N(\theta)$ quadratic
  equations and $n+1-d(\theta)-N(\theta)$ linear equations.\smallskip

  \noindent
  {\rm(b)} For all vectors $u$ in a dense open subset of $\bigcup
  X_\theta$ we have $X_\theta(u) \cong X_{\wb\theta}$.
\end{lemma}
\begin{proof}
  The result is clear unless $\theta \neq \emptyset$.  Using
  Lemma~\ref{lem:chopB} we may also assume that for all integers $a,b
  \geq 0$ with $0<a+b<n$ we have $\la_a \leq n-a-b$ or $\mu_b \leq
  n-a-b$.  This implies that $\theta$ has exactly one component, so
  $N(\theta)=1$.  Given any vector $u = (x_1,\dots,x_{2n+1}) \in
  \C^{2n+1}$ we will write $u_i = (x_1,\dots,x_i,0,\dots,0) \in
  \C^{2n+1}$ and $u'_i = (0,\dots,0,x_{i+1},\dots,x_{2n+1}) \in
  \C^{2n+1}$ for its projections to $E_i$ and $E^\op_{2n+1-i}$.
  
  Assume first that $\ell(\la) + \ell(\mu) < n$.  In this case
  $\theta$ intersects the diagonal, and since we also have $\la_1<n$
  and $\mu_1<n$, it follows that $d(\theta)=n$.  We will show that
  $\bigcup X_\theta$ is the quadric $\{ u \in \C^{2n+1} \mid (u,u)=0
  \} \subset \C^{2n+1}$.  Let $u = (x_1,\dots,x_{2n+1}) \in \C^{2n+1}$
  be any vector such that $(u,u)=0$, $x_i \neq 0$ for $1 \leq i \leq
  2n+1$, and $(u_i,u'_i) \neq 0$ for $1 \leq i \leq n$.  It is enough
  to show that $X_\theta(u) \cong X_{\wb\theta}$.  Set $\wb E =
  \bspan{u}^\perp/\bspan{u}$ and define isotropic flags in this vector
  space by $\wb E_i = ((E_{i+1} + \bspan{u}) \cap
  \bspan{u}^\perp)/\bspan{u}$ and $\wb E^\op_i = ((E^\op_{i+1} +
  \bspan{u}) \cap \bspan{u}^\perp)/\bspan{u}$.  For each $i<n$ we have
  $(E_{i+1} + \bspan{u}) \cap ((E^\op_{i+1})^\perp + \bspan{u}) =
  \bspan{u_{i+1},u'_{i+1}}$.  By the choice of $u$ we have
  $(u_{i+1},u'_{i+1}) \neq 0$, which implies that $\wb E_i \cap (\wb
  E^\op_i)^\perp = 0$.  Similarly we obtain $(\wb E_i)^\perp \cap
  \wb E^\op_i = 0$, so the flags $\wb E_\bull$ and $\wb E^\op_\bull$ are
  opposite.  Identify $\wb X = \OG(n-1,\wb E)$ with the set of point
  $V \in X$ for which $u \in V$.  Then we have $X^\la(E_\bull) \cap
  \wb X = \wb X^\la(\wb E_\bull)$ and $X^\mu(E^\op_\bull) \cap \wb X =
  \wb X^\mu(\wb E^\op_\bull)$, so $X_\theta(u) = X_\theta \cap \wb X =
  \wb X^\la(\wb E_\bull) \cap \wb X^\mu(\wb E^\op_\bull) = \wb
  X_{\wb\theta}$.
  
  Otherwise we have $\ell(\la) + \ell(\mu) = n$ and $\la_{\ell(\la)} =
  \mu_{\ell(\mu)} = 1$.  This implies that $\theta$ is disjoint from
  the diagonal, so $d(\theta) = n-1$.  For any point $V \in X_\theta$
  we must have $\dim(V \cap E_n) \geq \ell(\la)$ and $\dim(V \cap
  E^\op_n) \geq \ell(\mu)$, so $V = (V \cap E_n) \oplus (V \cap
  E^\op_n)$.  It follows that every vector $u = (x_1,\dots,x_{2n+1})
  \in \bigcup X_\theta$ satisfies $x_{n+1}=0$ and $(u,u) = 0$.  We
  will show that $\bigcup X_\theta$ is the complete intersection in
  $\C^{2n+1}$ defined by these two equations.  Let $u \in \C^{2n+1}$
  be any vector such that $x_{n+1}=0$, $(u,u)=0$, $x_i \neq 0$ for $i
  \neq n+1$, and $(u_i,u'_i)\neq 0$ for $1 \leq i \leq n-1$.  It is
  enough to show that $X_\theta(u) \cong X_{\wb\theta}$.  Set $\wb E =
  U^\perp/U$ where $U = \bspan{u_n,u'_n}$, and define isotropic flags
  in $\wb E$ by $\wb E_i = ((E_{i+1} + U) \cap U^\perp)/U$ and $\wb
  E^\op_i = ((E^\op_{i+1} + U) \cap U^\perp)/U$ for $1 \leq i \leq
  n-2$.  For each $i \leq n-2$ we have $(E_{i+1} + U) \cap
  ((E^\op_{i+1})^\perp + U) = \bspan{u_n,u'_n,u_{i+1},u'_{i+1}}$, so
  our choice of $u$ implies that $\wb E_i \cap (\wb E^\op_i)^\perp =
  0$.  Identify $\wb X = \OG(n-2,\wb E)$ with the set of points $V \in
  X$ for which $U \subset V$.  Then we have $X^\la(E_\bull) \cap \wb X
  = \wb X^{\wb\la}(\wb E_\bull)$ and $X^\mu(E^\op_\bull) \cap \wb X =
  \wb X^{\wb\mu}(\wb E^\op_\bull)$, where $\wb\la =
  (\la_1-1,\la_2-1,\dots,\la_{\ell(\la)}-1)$ and $\wb\mu =
  (\mu_1-1,\mu_2-1,\dots,\mu_{\ell(\mu)}-1)$.  We conclude that
  $X_\theta(u) = X_\theta \cap \wb X = \wb X^{\wb\la}(\wb E_\bull)
  \cap \wb X^{\wb\mu}(\wb E^\op_\bull) = \wb X_{\wb\theta}$, as
  required.
\end{proof}

Define a function $h : \N \times \Z \to \Z$ by
\begin{equation}\label{eqn:hdef}
  h(a,b) = \sum_{j=0}^b\, (-1)^j\, 2^{a-j}\, \binom{a}{j} \,.
\end{equation}
Here we set $\binom{a}{j} = 0$ unless $0 \leq j \leq a$.  Notice that
for $b \geq a$ we have $h(a,b) = (2-1)^a = 1$.  The binomial identity
implies that
\begin{equation}\label{eqn:hrel}
  h(a+1,b) + h(a,b-1) = 2\, h(a,b) \,.
\end{equation}

\begin{prop}\label{prop:tripleB}
  Let $\theta$ be a shifted skew diagram contained in $\rho_n$ and $0
  \leq p \leq n$.  Then we have $\euler{X}(\cO_\theta \cdot \cO^p) =
  h(N^-(\theta), d(\theta)-p)$.
\end{prop}
\begin{proof}
  Notice that $\euler{X}(\cO_\theta \cdot \cO^0) =
  \euler{X}(\cO_\theta) = 1$ and $\euler{X}(\cO_\emptyset \cdot \cO^p)
  = \delta_{p,0}$ as claimed, so we may assume that $\theta \neq
  \emptyset$ and $p \geq 1$.  Let $S \subset \P^{2n}$ be the quadric
  of isotropic lines in $\C^{2n+1}$.  The subvariety $\P(E_{n+1-p})
  \subset S$ defines the class $[\cO_{\P(E_{n+1-p})}] \in K(S)$.  Let
  $Z = \OF(1,n;2n+1)$ be the variety of two-step flags $L \subset V
  \subset \C^{2n+1}$ such $\dim(L)=1$ and $V \in X$, and let $\pi_1 :
  Z \to S$ and $\pi_n : Z \to X$ be the projections.  Since $\pi_n :
  \pi_1^{-1}(\P(E_{n+1-p})) \to X^p$ is a birational isomorphism and
  $\pi_1$ is flat, we obtain ${\pi_n}_* \pi_1^* [\cO_{\P(E_{n+1-p})}]
  = \cO^p \in K(X)$.  Since $\bigcup X_\theta$ is the affine cone over
  $\pi_1(\pi_n^{-1}(X_\theta))$, it follows from
  Lemma~\ref{lem:fiberB}(a) that the later variety has rational
  singularities.  Lemma~\ref{lem:gysin} and Lemma~\ref{lem:fiberB}(b)
  therefore imply that ${\pi_1}_* \pi_n^* \cO_\theta =
  [\cO_{\pi_1(\pi_n^{-1}(X_\theta))}] \in K(S)$, so it follows from
  the projection formula that
  \begin{equation*}
  \euler{X}(\cO_\theta \cdot \cO^p) =
  \euler{S}([\cO_{\pi_1(\pi_n^{-1}(X_\theta))}] \cdot
  [\cO_{\P(E_{n+1-\pi_n})}]) \,.
  \end{equation*}
  Let $Y \subset \P^{2n}$ be a complete intersection defined by the
  same equations as define $\bigcup X_\theta$ in $\C^{2n+1}$, except
  that one of the quadratic equations are omitted.  Since
  $\pi_1(\pi_n^{-1}(X_\theta)) = S \cap Y$ is a proper intersection of
  Cohen-Macaulay varieties, we obtain $[\cO_{\pi_1(\pi_n{-1}(X_\theta))}] =
  \iota^* [\cO_Y] \in K(S)$ where $\iota : S \subset \P^{2n}$ is the
  inclusion.  Another application of the projection formula yields
  \begin{equation*}
  \euler{X}(\cO_\theta \cdot \cO^p) =
  \euler{\P^{2n}}([\cO_Y] \cdot [\cO_{\P(E_{n+1-p})}]) \,.
  \end{equation*}
  Since the Grothendieck class of a quadric in $\P^{2n}$ is equal to
  $2t-t^2 \in K(\P^{2n}) = \Z[t]/(t^{2n+1})$, it follows from
  Lemma~\ref{lem:fiberB} that
  \begin{equation*}
  [\cO_Y] \cdot [\cO_{\P(E_{n+1-p})}] =
  t^{2n-d(\theta)+p} \cdot (2-t)^{N^-(\theta)} \,.
  \end{equation*}
  The proposition follows from this.
\end{proof}

The dual Schubert classes in the $K$-theory of $X = \OG(n,2n+1)$ are
defined by
\begin{equation}\label{eqn:dualB}
  \cO_\nu^* = \sum_{\nu/\tau\text{ rook strip}} (-1)^{|\nu/\tau|}\,
  \cO_\tau
\end{equation}
where the sum is over all strict partitions $\tau$ contained in $\nu$
such that the shifted skew diagram $\nu/\tau$ is a rook strip.  The
proof of Lemma~\ref{lem:dualA} shows that $\euler{X}(\cO_\nu^* \cdot
\cO^\mu) = \delta_{\nu,\mu}$.

A south-east corner of the shifted skew shape $\theta$ is any box $B$
of $\theta$ such that $\theta$ does not contain a box directly below
or directly to the right of $B$.  Let $\theta'$ be the skew shape
obtained by removing the south-east corners of $\theta$.  For $p \in
\Z$ we define
\begin{equation}\label{eqn:hsumB}
  \cB(\theta,p) = \sum_{\theta' \subset \varphi \subset \theta}
  (-1)^{|\theta|-|\varphi|}\, h(N^-(\varphi), d(\varphi)-p)
\end{equation}
where the sum is over all shifted skew diagrams $\varphi$ obtained by
removing a subset of the south-east corners from $\theta$.
Proposition~\ref{prop:tripleB} implies the following.

\begin{cor}
  Let $\la \subset \nu$ be strict partitions with $\nu_1 \leq n$ and
  let $0 \leq p \leq n$.  Then $c^\nu_{\la,p} = \cB(\nu/\la,p)$.
\end{cor}

\begin{lemma}\label{lem:vanishB}
  Let $\theta$ be a non-empty shifted skew diagram.\smallskip

  \noindent{\rm(a)} If $\theta$ is not a rim then $\cB(\theta,p)=0$
  for all $p$.\smallskip

  \noindent{\rm(b)} If $\theta$ contains a row or column with $a$
  boxes, then $\cB(\theta,p)=0$ for $p<a$.\smallskip
  
  \noindent{\rm(c)} If $\theta$ is a rook strip, then
  $\cB(\theta,1)=(-1)^{|\theta|-1}$.
\end{lemma}
\begin{proof}
  Assume the conditions of (a), (b), or (c) are met, and let $B \in
  \theta\ssm\theta'$ be a south-east corner of $\theta$; if $\theta$
  is not a rim, then choose $B$ such that $\theta$ contains a box
  strictly north and strictly west of $\theta$.  We claim that if
  $\varphi$ is any shifted skew diagram such that $\theta' \subset
  \varphi \subset \theta$, $\varphi \neq \emptyset$, and $B \not\in
  \varphi$, then
  \begin{equation*} 
  h(N^-(\varphi),d(\varphi)-p) = h(N^-(\varphi \cup
  B),d(\varphi \cup B)-p) \,;
  \end{equation*}
  in the situation of (c) we set $p=1$.  If $\theta$ is not a rim,
  then this is true because $N^-(\varphi) = N^-(\varphi\cup B)$ and
  $d(\varphi) = d(\varphi\cup B)$.  If the conditions for (b) hold,
  then we must have $N^-(\varphi) \leq d(\varphi)-a+1$ since $\varphi$
  contains a row or column with $a-1$ boxes, and this implies that
  $h(N^-(\varphi),d(\varphi)-p) = h(N^-(\varphi\cup B), d(\varphi\cup
  B)-p) = 1$ for $p < a$.  For (c), notice that $N^-(\varphi) =
  d(\varphi)-1$ whenever $\varphi$ is a non-empty rook strip.

  The lemma is immediate from the claim if $\theta$ is not a rook
  strip, and otherwise we obtain
  \begin{equation*}
  \begin{split}
    \cB(\theta,p) 
    &= (-1)^{|\theta|-1} 
    \left( h(N^-(B),d(B)-p) - h(N^-(\emptyset),d(\emptyset)-p) \right) \\
    &= (-1)^{|\theta|-1}
    \left( h(0,1-p) - h(0,-p) \right) .
  \end{split}
  \end{equation*}
  This proves the lemma when $\theta$ is a rook strip.
\end{proof}

Before we state our recursive Pieri rule for the coefficients
$\cB(\theta,p)$, we prove that the constants $c^{\nu^\vee}_{\la\mu}$
for maximal orthogonal Grassmannians are invariant under arbitrary
permutations of the partitions $\la$, $\mu$, $\nu$.  The same symmetry
is satisfied for Grassmannians of type A (for the same reason, see
\cite[Cor.~1]{buch:combinatorial}), but Example~\ref{ex:lgsym} below
shows that it fails for Lagrangian Grassmannians.

\begin{cor}
  The $K$-theoretic structure constants $c^\nu_{\la,\mu}$ for
  $\OG(n,2n+1)$ satisfy the symmetry $c^\nu_{\la,\mu} =
  c^{\mu^\vee}_{\la,\nu^\vee}$.
\end{cor}
\begin{proof}
  Lemma~\ref{lem:vanishB} implies that $\cO_\nu^* = (1-\cO^1)\cdot
  \cO^{\nu^\vee}$, so we obtain the identity $c^\nu_{\la,\mu} =
  \euler{X}(\cO^\la \cdot \cO^\mu \cdot \cO^{\nu^\vee} \cdot
  (1-\cO^1))$.  The corollary follows from this.
\end{proof}

A shifted skew diagram is called a {\em row\/} if all its boxes belong
to the same row, and a {\em column\/} if all boxes belong to the same
column.  If $\theta$ is a non-empty rim, then we define the {\em
  north-east arm\/} of $\theta$ to be the largest row or column that
can be obtained by intersecting $\theta$ with a square whose
upper-right box agrees with the upper-right box of $\theta$.  We let
$\wh\theta$ be the diagram obtained by removing the north-east arm
from $\theta$.  Notice that $\theta$ is a row or a column if and only
if $\wh\theta = \emptyset$.
\begin{equation*}
\pic{.5}{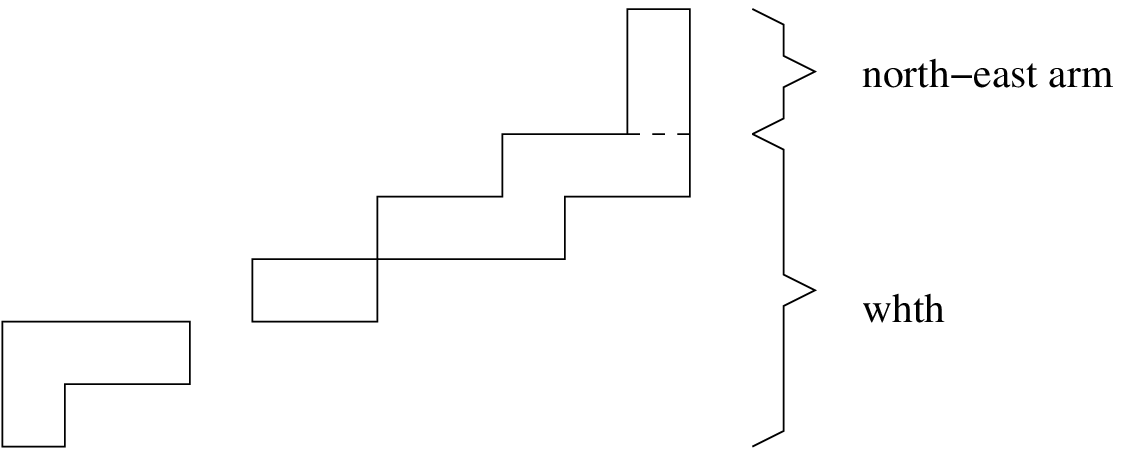}
\end{equation*}

The following theorem is our recursive Pieri rule for the $K$-theory
of maximal orthogonal Grassmannians.  It implies that the Pieri
coefficients $c^\nu_{\la,p} = \cB(\nu/\la,p)$ have alternating signs,
$(-1)^{|\nu/\la|-p}\, c^\nu_{\la,p} \geq 0$.

\begin{thm}\label{thm:pieriBrec}
  Let $\theta$ be a shifted skew diagram and let $p \in \Z$.  If
  $\theta$ is not a rim then $\cB(\theta,p) = 0$.  If $p\leq 0$ then
  $\cB(\theta,p) = \delta_{\theta,\emptyset}$, and $\cB(\emptyset,p)$
  is equal to one if $p \leq 0$ and is zero otherwise.  If $\theta
  \neq \emptyset$ is a rim and $p>0$, then $\cB(\theta,p)$ is
  determined by the following rules, with $a =
  |\theta|-|\wh\theta|$.\smallskip
  
  \noindent{\rm(i)} If $\theta$ is a row or a column then
  $\cB(\theta,p) = \delta_{|\theta|,p}$.\smallskip

  \noindent{\rm(ii)} If $\wh\theta \neq \emptyset$ and the north-east
  arm of $\theta$ is connected to $\wh\theta$, then we have $\cB(\theta,p) =
  \cB(\wh\theta,p-a) - \cB(\wh\theta,p-a+1)$.\smallskip

  \noindent{\rm(iii)} Assume that $\wh\theta \neq \emptyset$ and the north-east
  arm of $\theta$ is not connected to $\wh\theta$.  If $p<a$ then
  $\cB(\theta,p) = 0$.  Otherwise we have
  \begin{equation*}
  \cB(\theta,p) = (2-\delta_{p,a})(\cB(\wh\theta,p-a) -
  \cB(\wh\theta,p-a+1)) +
  (1-\delta_{a,1})(\cB(\wh\theta,p-a+2)-\cB(\wh\theta,p-a+1)).
  \end{equation*}
\end{thm}
\begin{proof}
  We may assume that $\theta$ is a non-empty rim and $p > 0$, since
  otherwise the theorem follows from Lemma~\ref{lem:vanishB}.  In
  particular we have $d(\theta) = |\theta|$.  If $\theta$ is a row or
  column, then $\cB(\theta,p) = h(0,|\theta|-p) - h(0,|\theta'|-p) =
  \delta_{p,|\theta|}$, as claimed in (i).  Otherwise let $\psi$ be
  the north-east arm of $\theta$, let $B \in \theta \ssm \theta'$ be
  the south-east corner farthest to the north, and set $\psi' = \psi
  \ssm B$ and $\wh\theta' = \wh\theta \cap \theta'$.  We have
  $\wh\theta \neq \emptyset$ and $a = |\psi|$.

  Assume that $\psi$ is connected to $\wh\theta$.  Then $\wh\theta$ is
  not a rook strip.  We first consider the case where $\psi$ is a row
  attached to the right side of the upper-right box of $\wh\theta$.
  For any rim $\varphi$ with $\wh\theta' \subset \varphi \subset
  \wh\theta$ we then have
  \begin{equation*}
  \begin{split}
  & (-1)^{|\theta|-|\varphi\cup\psi|} \left(
    h(N^-(\varphi\cup\psi), |\varphi\cup\psi|-p) -
    h(N^-(\varphi\cup\psi'), |\varphi\cup\psi'|-p)\right) \\
  &=
  (-1)^{|\wh\theta|-|\varphi|} \left(
    h(N^-(\varphi),|\varphi|-p+a) -
    h(N^-(\varphi),|\varphi|-p+a-1)\right) .
  \end{split}
  \end{equation*}
  This implies that $\cB(\theta,p) = \cB(\wh\theta,p-a) -
  \cB(\wh\theta,p-a+1)$ by summing over $\varphi$.  Next consider the
  case where $\psi$ is a column attached above the upper-right box of
  $\wh\theta$.  For any diagram $\varphi$ with $\wh\theta' \subset
  \varphi \subset \wh\theta \ssm B$, it follows from (\ref{eqn:hrel})
  that
  \begin{equation*}
  \begin{split}
    &(-1)^{|\theta|-|\varphi \cup \psi|} \left(
      h(N^-(\varphi\cup\psi),|\varphi\cup\psi|-p) 
      - h(N^-(\varphi \cup B \cup \psi), 
      |\varphi \cup B \cup \psi|-p) \right) \\
    &= 
    (-1)^{|\wh\theta|-|\varphi|} \left(
      h(N^-(\varphi)+1,|\varphi|-p+a) 
      - h(N^-(\varphi\cup B),|\varphi\cup B|-p+a) \right) \\
    &= 
    (-1)^{|\wh\theta|-|\varphi|} \left(
      h(N^-(\varphi),|\varphi|-p+a)
      - h(N^-(\varphi\cup B),|\varphi\cup B|-p+a) \right)\\
      & \hspace{15mm}
      - h(N^-(\varphi),|\varphi|-p+a-1) 
      + h(N^-(\varphi\cup B),|\varphi\cup B|-p+a-1)
  \end{split}
  \end{equation*}
  This again implies that $\cB(\theta,p) = \cB(\wh\theta,p-a) -
  \cB(\wh\theta,p-a+1)$, as required by (ii).

  Now assume that $\psi$ is not connected to $\wh\theta$.  We may also
  assume that $p\geq a$, since otherwise $\cB(\theta,p)=0$ by
  Lemma~\ref{lem:vanishB}(b).  We first consider the case $a>1$.  For
  any non-empty rim $\varphi$ with $\wh\theta' \subset \varphi \subset
  \wh\theta$, we obtain
  \begin{equation*}
  \begin{split}
    & (-1)^{|\theta|-|\varphi\cup\psi|} \left(
      h(N^-(\varphi\cup\psi),|\varphi\cup\psi|-p) -
      h(N^-(\varphi\cup\psi'),|\varphi\cup\psi'|-p) \right) \\
    &= (-1)^{|\wh\theta|-|\varphi|} \left(
      h(N^-(\varphi)+1,|\varphi|-p+a) -
      h(N^-(\varphi)+1,|\varphi|-p+a-1) \right) \\
    &= (-1)^{|\wh\theta|-|\varphi|} (
    2\, h(N^-(\varphi),|\varphi|-p+a) 
    - 3\, h(N^-(\varphi),|\varphi|-p+a-1) \\
    & \hspace{23mm} + h(N^-(\varphi),|\varphi|-p+a-2) ) \,.
  \end{split}
  \end{equation*}
  If $\wh\theta$ is not a rook strip, then $\wh\theta' \neq \emptyset$
  and we obtain that $\cB(\theta,p) = 2\, \cB(\wh\theta,p-a) - 3\,
  \cB(\wh\theta,p-a+1) + \cB(\wh\theta,p-a+2)$.  If $p>a$ then this is
  the identity of (iii), and if $p=a$ then Lemma~\ref{lem:vanishB}(b)
  shows that $\cB(\wh\theta,p-a) = \cB(\wh\theta,p-a+1) = 0$, so (iii)
  also holds in this case.  If $\wh\theta$ is a rook strip, then
  summing over $\varphi \neq \emptyset$ gives the identity
  \begin{equation*}
  \begin{split}
    \cB(\theta,p) \ 
    & = \ 2\, \cB(\wh\theta,p-a) - 3\, \cB(\wh\theta,p-a+1)
    + \cB(\wh\theta,p-a+2) \\
    & \ \ \ \ - (-1)^{|\wh\theta|}\left( h(0,a-p) - 2 h(0,a-p-1) +
      h(0,a-p-2) \right) \,.
  \end{split}
  \end{equation*}
  If $p>a$ then the three last terms vanish, and if $p=a$ then
  Lemma~\ref{lem:vanishB} implies that $\cB(\theta,p) =
  \cB(\wh\theta,2) - 2\, \cB(\wh\theta,1)$, as required.  At last we
  consider the case $a=1$.  Then $\psi' = \emptyset$, so for any
  non-empty rim $\varphi$ with $\wh\theta' \subset \varphi \subset
  \wh\theta$ we get
  \begin{equation*}
  \begin{split}
    & (-1)^{|\theta|-|\varphi\cup\psi|} \left(
      h(N^-(\varphi\cup\psi),|\varphi\cup\psi|-p) -
      h(N^-(\varphi\cup\psi'),|\varphi\cup\psi'|-p) \right) \\
    &= (-1)^{|\wh\theta|-|\varphi|} \left(
      h(N^-(\varphi)+1,|\varphi|-p+1) -
      h(N^-(\varphi),|\varphi|-p) \right) \\
    &= (-1)^{|\wh\theta|-|\varphi|} \left(
    2\, h(N^-(\varphi),|\varphi|-p+1) 
    - 2\, h(N^-(\varphi),|\varphi|-p) \right) .
  \end{split}
  \end{equation*}
  If $\wh\theta$ is not a rook strip, then this implies that
  $\cB(\theta,p) = 2\, \cB(\wh\theta,p-1) - 2\, \cB(\wh\theta,p)$, and
  for $p=a$ we have $\cB(\wh\theta,p-1) = \cB(\wh\theta,p) = 0$ by
  Lemma~\ref{lem:vanishB}.  Finally, if $\wh\theta$ is a rook strip,
  then we obtain $\cB(\theta,p) = 2\, \cB(\wh\theta,p-1) - 2\,
  \cB(\wh\theta,p) - (-1)^{|\wh\theta|} h(0,1-p)$, which agrees with
  (iii) by Lemma~\ref{lem:vanishB}.  This establishes (iii) and
  completes the proof.
\end{proof}

Let $\theta$ be a rim.  A {\em KOG-tableau\/} of shape $\theta$ is a
labeling of the boxes of $\theta$ with positive integers such that (i)
each row of $\theta$ is strictly increasing from left to right; (ii)
each column of $\theta$ is strictly increasing from top to bottom; and
(iii) each box is either smaller than or equal to all the boxes
south-west of it, or it is greater than or equal to all the boxes
south-west of it.  If $\theta$ is not a rim, then there are no
KOG-tableaux with shape $\theta$.  The content of a KOG-tableau is the
set of integers contained in its boxes.

\begin{cor}\label{cor:pieriBtab}
  The constant $c^\nu_{\la,p}$ for $K(\OG(n,2n+1))$ is equal to
  $(-1)^{|\nu/\la|-p}$ times the number of KOG-tableaux of shape
  $\nu/\la$ with content $\{1,2,\dots,p\}$.
\end{cor}
\begin{proof}
  Define $\wt\cB(\theta,p)$ to be $(-1)^{|\theta|-p}$ times the number
  of KOG-tableaux of shape $\theta$ with content $\{1,2,\dots,p\}$.
  It is enough to show that these integers satisfy the recursions
  given by Theorem~\ref{thm:pieriBrec}.  We check this when the
  north-east arm $\psi$ is a column that is disconnected from
  $\wh\theta$, $\wh\theta \neq \emptyset$, and $1 < a < p$ where $a =
  |\theta|-|\wh\theta|$.  The remaining cases are left to the reader.

  In any KOG-tableau of shape $\theta$, the content of (the boxes of)
  $\psi$ must be either $S_1 = \{1,2,\dots,a\}$ or $S_2 =
  \{1,2,\dots,a-1,p\}$.  If the content of $\psi$ is $S_1$, then
  $\wh\theta$ must be a KOG-tableau whose content is either
  $\{a+1,a+2,\dots,p\}$ or $\{a,a+1,\dots,p\}$.  If the content of
  $\psi$ is $S_2$, then $\wh\theta$ must be a KOG-tableau whose
  content is one of the sets $\{a,a+1,\dots,p-1\}$,
  $\{a-1,a,\dots,p-1\}$, $\{a,a+1,\dots,p\}$, or $\{a-1,a,\dots,p\}$.
  This gives the identity $\wt\cB(\theta,p) = 2\,
  \wt\cB(\wh\theta,p-a) - 3\, \wt\cB(\wh\theta,p-a+1) +
  \wt\cB(\wh\theta,p-a+2)$, as required.
\end{proof}

Itai Feigenbaum and Emily Sergel have shown us a proof that
KOG-tableaux are invariant under Thomas and Yong's jeu-de-taquin
slides \cite{thomas.yong:jeu}.  Corollary~\ref{cor:pieriBtab}
therefore confirms that the Pieri coefficients $c^\nu_{\la,p}$ are
correctly computed by Thomas and Yong's conjectured
Littlewood-Richardson rule.

\begin{example}
  Let $\la = (6,4,1)$ and $\nu = (7,6,3,1)$.  Then the constant
  $c^\nu_{\la,5} = -7$ for $K(\OG(7,15))$ is counted by the
  KOG-tableaux displayed below.\medskip

  \noindent \ \ 
  $\tableau{10}{&&&{1}\\&&{1}&{5}\\{2}&{4}\\{3}}$ \ \ 
  $\tableau{10}{&&&{1}\\&&{2}&{5}\\{2}&{4}\\{3}}$ \ \ 
  $\tableau{10}{&&&{1}\\&&{2}&{5}\\{3}&{5}\\{4}}$ \ \ 
  $\tableau{10}{&&&{1}\\&&{2}&{5}\\{3}&{4}\\{4}}$ \ \ 
  $\tableau{10}{&&&{1}\\&&{4}&{5}\\{1}&{3}\\{2}}$ \ \ 
  $\tableau{10}{&&&{1}\\&&{4}&{5}\\{2}&{4}\\{3}}$ \ \ 
  $\tableau{10}{&&&{1}\\&&{4}&{5}\\{2}&{3}\\{3}}$
\end{example}

\section{Lagrangian Grassmannians}\label{sec:typeC}

Let $e_1,\dots,e_{2n}$ be the standard basis of $\C^{2n}$ and define
a symplectic form by
\begin{equation*}
(e_i,e_j) = \begin{cases}
  1 & \text{if $i+j=2n+1$ and $i<j$;} \\
  -1& \text{if $i+j=2n+1$ and $i>j$; and} \\
  0 & \text{if $i+j\neq 2n+1$.}
\end{cases}
\end{equation*}
Let $X = \LG(n,2n) = \{ V \subset \C^{2n} \mid \dim(V)=n \text{ and }
(V,V)=0 \}$ be the Lagrangian Grassmannian of maximal isotropic
subspaces of $\C^{2n}$.  This is a non-singular variety of dimension
$\binom{n+1}{2}$.  The Schubert varieties in $X$ are indexed by strict
partitions $\la = (\la_1 > \la_2 > \dots > \la_\ell > 0)$ with $\la_1
\leq n$, the same partitions as are used for the orthogonal
Grassmannian $\OG(n,2n+1)$.  The notation introduced for shifted skew
diagrams in section~\ref{sec:typeB} will also be used in this section.
The Schubert variety for the strict partition $\la$ relative to an
isotropic flag $0 \subset F_1 \subset F_2 \subset \dots \subset F_n
\subset \C^{2n}$ with $(F_n,F_n)=0$ is defined by
\begin{equation*} 
X^\la(F_\bull) = \{ V \in X \mid 
\dim(V \cap F_{n+1-\la_i}) \geq i ~\forall 1 \leq i \leq \ell(\la) \}\,.
\end{equation*}
The standard flags are defined by $E_i = \bspan{e_1,\dots,e_i}$ and
$E^\op_i = \bspan{e_{2n+1-i},\dots,e_{2n}}$.  We have $X^\la(E_\bull)
\cap X^\mu(E^\op_\bull) \neq \emptyset$ if and only if $\la \subset
\mu^\vee$, where $\mu^\vee$ is the $n$--dual of $\mu$.  Set $X^\la =
X^\la(E_\bull)$.  The classes $\cO^\la = [\cO_{X^\la}]$ form a
$\Z$-basis for the Grothendieck ring $K(X)$, which is generated as a
$\Z$-algebra by the special classes $\cO^1,\dots,\cO^n$.  A shifted
skew diagram $\theta = \mu^\vee/\la$ defines the Richardson variety
$X_\theta = X^\la(E_\bull) \cap X^\mu(E^\op_\bull)$ and the class
$\cO_\theta = [\cO_{X_\theta}] = \cO^\la \cdot \cO^\mu \in K(X)$.  For
any vector $u \in \C^{2n}$ we set $X_\theta(u) = \{ V \in X_\theta
\mid u \in V \}$, and define $\bigcup X_\theta = \bigcup_{V\in
  X_\theta} V \subset \C^{2n}$.

Assume that $a,b \geq 0$ are integers such that $0 < a+b < n$, $\la_a
> n-a-b$, and $\mu_b > n-a-b$.  Set $\la' = (\la_1+a+b-n, \dots,
\la_a+a+b-n)$, $\mu' = (\mu_1+a+b-n, \dots, \mu_b+a+b-n)$, and
$\theta' = {\mu'}^\vee/\la'$, where ${\mu'}^\vee$ is the $(a+b)$--dual
of $\mu'$.  Then $\theta'$ defines a Richardson variety $X'_{\theta'}$
in $X' = \LG(a+b,E')$, where $E' = E_{a+b}\oplus E^\op_{a+b}$ with
basis $e_1,\dots,e_{a+b},e_{2n+1-a-b},\dots,e_{2n}$.  Similarly we set
$\la'' = (\la_{a+1},\dots,\la_{\ell(\la)})$, $\mu'' =
(\mu_{b+1},\dots,\mu_{\ell(\mu)})$, and $\theta'' =
{\mu''}^\vee/\la''$, using the $(n-a-b)$--dual of $\mu''$.  This
defines the Richardson variety $X''_{\theta''}$ in $X'' = \LG(n-a-b,
E'')$, where $E'' = {E'}^\perp = \bspan{e_{a+b+1},\dots,e_{2n-a-b}}$.

\begin{lemma}\label{lem:chopC}
  {\rm(a)} We have $\bigcup X_\theta = \bigcup X'_{\theta'} \times
  \bigcup X''_{\theta''}$ in $\C^{2n} = E' \times E''$.\smallskip

  \noindent
  (b) For arbitrary vectors $u' \in E'$ and $u'' \in E''$, the natural
  inclusion $X' \times X'' \subset X$ defined by $(V',V'') \mapsto
  V'\oplus V''$ identifies $X'_{\theta'}(u') \times
  X''_{\theta''}(u'')$ with $X_\theta(u'+u'')$.
\end{lemma}
\begin{proof}
  If $V \in X_\theta$ then $\dim(V \cap E_{a+b}) \geq \dim(V \cap
  E_{n+1-\la_a}) \geq a$ and $\dim(V \cap E^\op_{a+b}) \geq \dim(V
  \cap E^\op_{n+1-\la_b}) \geq b$.  This implies that $V' = V \cap E'$
  is a maximal isotropic subspace in $E'$, so $V' \in X'$.  It also
  follows that $\dim(V \cap E'') = \dim(V \cap {E'}^\perp) = n-a-b$,
  so $V'' = V \cap E'' \in X''$.  Given arbitrary points $V' \in X'$
  and $V'' \in X''$, it follows from the definitions that $V'\oplus
  V'' \in X_\theta(u'+u'')$ if and only $V' \in X'_{\theta'}(u')$ and
  $V'' \in X''_{\theta''}(u'')$.  The lemma follows from this.
\end{proof}

\begin{lemma}\label{lem:fiberC}
  {\rm(a)} The set $\bigcup X_\theta \subset \C^{2n}$ is a scheme
  theoretic complete intersection with rational singularities.  It has
  dimension $n+d(\theta)$ and is defined by $N'(\theta)$ quadratic
  equations and $n-d(\theta)-N'(\theta)$ linear equations.\smallskip

  \noindent
  {\rm(b)} For all vectors $u$ in a dense open subset of $\bigcup
  X_\theta$ we have $X_\theta(u) \cong X_{\wb\theta}$.
\end{lemma}
\begin{proof}
  Using Lemma~\ref{lem:chopC} we may assume that for all integers $a,b
  \geq 0$ with $0 < a+b < n$ we have $\la_a \leq n-a-b$ or $\mu_b \leq
  n-a-b$.  This implies that $\theta$ has exactly one component.
  Given any vector $u = (x_1,\dots,x_{2n}) \in \C^{2n}$ we will write
  $u_i = (x_1,\dots,x_i,0,\dots,0) \in \C^{2n}$ and $u'_i =
  (0,\dots,0,x_{i+1},\dots,x_{2n}) \in \C^{2n}$ for its projections to
  $E_i$ and $E^\op_{2n-i}$.
  
  Assume first that $\ell(\la) + \ell(\mu) < n$.  In this case
  $\theta$ intersects the diagonal, so $N'(\theta) = 0$.  Since
  $\la_1<n$ and $\mu_1<n$ we also have $d(\theta)=n$.  Let $u =
  (x_1,\dots,x_{2n}) \in \C^{2n}$ be any vector such that $x_i \neq 0$
  and $(u_i,u'_i) \neq 0$ for $1 \leq i \leq 2n$.  It is enough to
  show that $X_\theta(u) \cong X_{\wb\theta}$.  Set $\wb E =
  u^\perp/\bspan{u}$ and define isotropic flags in this vector space
  by $\wb E_i = ((E_{i+1} + \bspan{u}) \cap
  \bspan{u}^\perp)/\bspan{u}$ and $\wb E^\op_i = ((E^\op_{i+1} +
  \bspan{u}) \cap \bspan{u}^\perp)/\bspan{u}$.  For each $i<n$ we have
  $(E_{i+1} + \bspan{u}) \cap ((E^\op_{i+1})^\perp + \bspan{u}) =
  \bspan{u_{i+1},u'_{i+1}}$.  By the choice of $u$ we have
  $(u_{i+1},u'_{i+1}) \neq 0$, which implies that $\wb E_i \cap (\wb
  E^\op_i)^\perp = 0$.  Similarly we obtain $(\wb E_i)^\perp \cap
  E^\op_i = 0$, so the flags $\wb E_\bull$ and $\wb E^\op_\bull$ are
  opposite.  Identify $\wb X = \LG(n-1,\wb E)$ with the set of point
  $V \in X$ for which $u \in V$.  Then we have $X^\la(E_\bull) \cap
  \wb X = \wb X^\la(\wb E_\bull)$ and $X^\mu(E^\op_\bull) \cap \wb X =
  \wb X^\mu(\wb E^\op_\bull)$, so $X_\theta(u) = X_\theta \cap \wb X =
  \wb X^\la(\wb E_\bull) \cap \wb X^\mu(\wb E^\op_\bull) = \wb
  X_{\wb\theta}$.
  
  Otherwise we have $\ell(\la) + \ell(\mu) = n$ and $\la_{\ell(\la)} =
  \mu_{\ell(\mu)} = 1$.  This implies that $\theta$ is disjoint from
  the diagonal, so $N'(\theta) = 1$ and $d(\theta) = n-1$.  For any
  point $V \in X_\theta$ we have $\dim(V \cap E_n) \geq \ell(\la)$ and
  $\dim(V \cap E^\op_n) \geq \ell(\mu)$, so $V = (V \cap E_n) \oplus
  (V \cap E^\op_n)$.  It follows that every vector $u \in \bigcup
  X_\theta$ satisfies $(u_n,u'_n) = 0$.  We will show that $\bigcup
  X_\theta \subset \C^{2n}$ is defined by this equation.  Let $u \in
  \C^{2n}$ be any vector such that $x_i \neq 0$ for $1 \leq i \leq
  2n$, $(u_n,u'_n)=0$, and $(u_i,u'_i)\neq 0$ for $1 \leq i \leq n-1$.
  It is enough to show that $X_\theta(u) \cong X_{\wb\theta}$.  Set
  $\wb E = U^\perp/U$ where $U = \bspan{u_n,u'_n}$, and define
  isotropic flags in $\wb E$ by $\wb E_i = ((E_{i+1} + U) \cap
  U^\perp)/U$ and $\wb E^\op_i = ((E^\op_{i+1} + U) \cap U^\perp)/U$
  for $1 \leq i \leq n-2$.  For each $i \leq n-2$ we have $(E_{i+1} +
  U) \cap ((E^\op_{i+1})^\perp + U) =
  \bspan{u_n,u'_n,u_{i+1},u'_{i+1}}$, so our choice of $u$ implies
  that $\wb E_i \cap (\wb E^\op_i)^\perp = 0$.  Identify $\wb X =
  \LG(n-2,\wb E)$ with the set of points $V \in X$ for which $U
  \subset V$.  Then we have $X^\la(E_\bull) \cap \wb X = \wb
  X^{\wb\la}(\wb E_\bull)$ and $X^\mu(E^\op_\bull) \cap \wb X = \wb
  X^{\wb\mu}(\wb E^\op_\bull)$, where $\wb\la =
  (\la_1-1,\la_2-1,\dots,\la_{\ell(\la)}-1)$ and $\wb\mu =
  (\mu_1-1,\mu_2-1,\dots,\mu_{\ell(\mu)}-1)$.  We conclude that
  $X_\theta(u) = X_\theta \cap \wb X = \wb X^{\wb\la}(\wb E_\bull)
  \cap \wb X^{\wb\mu}(\wb E^\op_\bull) = \wb X_{\wb\theta}$, as
  required.
\end{proof}

\begin{prop}\label{prop:tripleC}
  For $0 \leq p \leq n$ we have $\euler{X}(\cO_\theta \cdot \cO^p) =
  h(N'(\theta), d(\theta)-p)$.
\end{prop}
\begin{proof}
  Let $Z = \IF(1,n;2n)$ be the variety of two-step flags $L \subset V
  \subset \C^{2n}$ such that $\dim(L)=1$ and $V \in X$, and let $\pi_1
  : Z \to \P^{2n-1}$ and $\pi_n : Z \to X$ be the projections.  Then
  ${\pi_n}_* \pi_1^* [\cO_{\P(E_{n+1-p})}] = \cO^p \in K(X)$, and
  Lemmas \ref{lem:fiberC} and \ref{lem:gysin} imply that ${\pi_1}_*
  \pi_n^* \cO_\theta = [\cO_{\pi_1(\pi_n^{-1}(X_\theta))}] \in
  K(\P^{2n-1})$.  The projection formula gives
  \begin{equation*}
    \euler{X}(\cO_\theta \cdot \cO^p) =
    \euler{\P^{2n-1}}([\cO_{\pi_1(\pi_n^{-1}(X_\theta))}] \cdot
    [\cO_{\P(E_{n+1-p})}]) \,.
  \end{equation*}
  Lemma~\ref{lem:fiberC} shows that
  $[\cO_{\pi_1(\pi_n^{-1}(X_\theta))}] \cdot [\cO_{\P(E_{n+1-p})}] =
  t^{2n-d(\theta)+p-1} \cdot (2-t)^{N'(\theta)} \in K(\P^{2n-1}) =
  \Z[t]/(t^{2n})$ and the proposition follows from this.
\end{proof}

Given a shifted skew diagram $\theta$ and $p \in \Z$ we define
\begin{equation}\label{eqn:hsumC}
  \cC(\theta,p) = 
  \sum_{\theta' \subset \varphi \subset \theta}
  (-1)^{|\theta|-|\varphi|}\, h(N'(\varphi), d(\varphi)-p)
\end{equation}
where the sum is over all shifted skew diagrams $\varphi$ obtained by
removing a subset of the south-east corners from $\theta$.  The dual
Schubert classes $\cO_\nu^*$ in the $K$-theory of $X = \LG(n,2n)$ is
defined by the same expression (\ref{eqn:dualB}) as for $\OG(n,2n+1)$.
The identity $\euler{X}(\cO_\nu^* \cdot \cO^\mu) = \delta_{\nu,\mu}$
and Proposition~\ref{prop:tripleC} imply the following.

\begin{cor}
  Let $\la \subset \nu$ be strict partitions with $\nu_1 \leq n$ and
  let $0 \leq p \leq n$.  Then $c^\nu_{\la,p} = \cC(\nu/\la, p)$.
\end{cor}

The following result is our recursive Pieri formula for Lagrangian
Grassmannians.

\begin{thm}\label{thm:pieriCrec}
  Let $\theta$ be a shifted skew diagram and let $p \in \Z$.  If
  $\theta$ is not a rim then $\cC(\theta,p)=0$.  If $p \leq 0$ then
  $\cC(\theta,p) = \delta_{\theta,\emptyset}$, and $\cC(\emptyset,p)$
  is equal to one if $p\leq 0$ and is zero otherwise.  If $\theta \neq
  \emptyset$ is a rim and $p>0$, then $\cC(\theta,p)$ is determined by
  the following rules, with $a = |\theta|-|\wh\theta|$.\smallskip

  \noindent{\rm(i)} If $\wh\theta = \emptyset$ and $\theta$ meets the
  diagonal, then we have $\cC(\theta,p) = \delta_{p,|\theta|} -
  \delta_{p,|\theta|-1}$ if $\theta$ is a column, and $\cC(\theta,p) =
  \delta_{p,|\theta|}$ otherwise.\smallskip

  \noindent{\rm(ii)} If $\wh\theta = \emptyset$ and $\theta$ is
  disjoint from the diagonal, then $\cC(\theta,p) = 2\,
  \delta_{p,|\theta|} - \delta_{p,|\theta|-1}$.\smallskip
  
  \noindent{\rm(iii)} If $\wh\theta \neq \emptyset$ and the north-east
  arm of $\theta$ is connected to $\wh\theta$, then we have
  $\cC(\theta,p) = \cC(\wh\theta,p-a) -
  \cC(\wh\theta,p-a+1)$.\smallskip
  
  \noindent{\rm(iv)} Finally assume that $\wh\theta \neq \emptyset$
  and the north-east arm of $\theta$ is not connected to $\wh\theta$.
  If $a=1$ then $\cC(\theta,p) = 2\, \cC(\wh\theta,p-a) - 2\,
  \cC(\wh\theta,p-a+1)$, while if $a>1$ we have $\cC(\theta,p) = 2\,
  \cC(\wh\theta,p-a) - 3\, \cC(\wh\theta,p-a+1) +
  \cC(\wh\theta,p-a+2)$.
\end{thm}
\begin{proof}
  If $\theta$ is not a rim, then let $B \in \theta$ be a south-east
  corner such that $\theta$ contains a box strictly north and strictly
  west of $B$.  For any shifted skew diagram $\varphi$ with $\theta'
  \subset \varphi \subset \theta \ssm \{B\}$ we obtain
  $h(N'(\varphi),d(\varphi)-p) = h(N'(\varphi \cup B),d(\varphi \cup
  B)-p)$ and $\cC(\theta,p)=0$.  We also have $\cC(\emptyset,p) =
  h(0,-p)$ which is equal to one when $p\leq 0$ and zero otherwise.
  If $\theta \neq \emptyset$ and $p \leq 0$, then $\cC(\theta,p) =
  \sum_{\theta' \subset \varphi \subset \theta}
  (-1)^{|\theta|-|\varphi|} = 0$.  If $\theta$ is a row or a column
  and $p > 0$, then $\cC(\theta,p) = h(N'(\theta),|\theta|-p) -
  h(N'(\theta'),|\theta'|-p)$, from which (i) and (ii) are easily
  checked.  We can therefore assume that $\theta$ is a rim, $\wh\theta
  \neq \emptyset$, and $p>0$.  Let $\psi$ be the north-east arm of
  $\theta$, let $B \in \theta\ssm\theta'$ be the south-east corner
  farthest to the north, and set $\psi' = \psi \ssm B$ and $\wh\theta'
  = \wh\theta \cap \theta'$.  We have $a = |\psi|$.  Let $\varphi$ be
  any rim such that $\wh\theta' \subset \varphi \subset \wh\theta \ssm
  B$.

  Assume that $\psi$ is connected to $\wh\theta$.  If $\psi$ is a is a
  row attached to the right side of the upper-right box of
  $\wh\theta$, then we have
  \begin{equation*}
    \begin{split}
      & (-1)^{|\theta|-|\varphi\cup\psi|} \left(
        h(N^-(\varphi\cup\psi), |\varphi\cup\psi|-p) -
        h(N^-(\varphi\cup\psi'), |\varphi\cup\psi'|-p)\right) \\
      &= (-1)^{|\wh\theta|-|\varphi|} \left(
        h(N^-(\varphi),|\varphi|-p+a) -
        h(N^-(\varphi),|\varphi|-p+a-1)\right) 
    \end{split}
  \end{equation*}
  which implies that $\cC(\theta,p) = \cC(\wh\theta,p-a) -
  \cC(\wh\theta,p-a+1)$ by summing over $\varphi$.  If $\psi$ is a
  column attached above the upper-right box of $\wh\theta$, then
  (\ref{eqn:hrel}) implies that
  \begin{equation*}
    \begin{split}
      &(-1)^{|\theta|-|\varphi \cup \psi|} \left(
        h(N^-(\varphi\cup\psi),|\varphi\cup\psi|-p) 
        - h(N^-(\varphi \cup B \cup \psi), 
        |\varphi \cup B \cup \psi|-p) \right) \\
      &= (-1)^{|\wh\theta|-|\varphi|} \left(
        h(N^-(\varphi),|\varphi|-p+a)
        - h(N^-(\varphi\cup B),|\varphi\cup B|-p+a) \right)\\
      & \hspace{15mm}
      - h(N^-(\varphi),|\varphi|-p+a-1) 
      + h(N^-(\varphi\cup B),|\varphi\cup B|-p+a-1)
    \end{split}
  \end{equation*}
  which again implies that $\cC(\theta,p) = \cC(\wh\theta,p-a) -
  \cC(\wh\theta,p-a+1)$, as required by (iii).
  
  Now assume that $\psi$ is not connected to $\wh\theta$.  If $a>1$
  then
  \begin{equation*}
    \begin{split}
      & (-1)^{|\theta|-|\varphi\cup\psi|} \left(
        h(N^-(\varphi\cup\psi),|\varphi\cup\psi|-p) -
        h(N^-(\varphi\cup\psi'),|\varphi\cup\psi'|-p) \right) \\
      &= (-1)^{|\wh\theta|-|\varphi|} (
      2\, h(N^-(\varphi),|\varphi|-p+a) 
      - 3\, h(N^-(\varphi),|\varphi|-p+a-1) \\
      & \hspace{23mm} + h(N^-(\varphi),|\varphi|-p+a-2) )
    \end{split}
  \end{equation*}
  while implies that $\cC(\theta,p) = 2\, \cC(\wh\theta,p-a) - 3\,
  \cC(\wh\theta,p-a+1) + \cC(\wh\theta,p-a+2)$.  And if $a=1$ then
  \begin{equation*}
    \begin{split}
      & (-1)^{|\theta|-|\varphi\cup\psi|} \left(
        h(N^-(\varphi\cup\psi),|\varphi\cup\psi|-p) -
        h(N^-(\varphi\cup\psi'),|\varphi\cup\psi'|-p) \right) \\
      &= (-1)^{|\wh\theta|-|\varphi|} \left(
        2\, h(N^-(\varphi),|\varphi|-p+1) 
        - 2\, h(N^-(\varphi),|\varphi|-p) \right) 
    \end{split}
  \end{equation*}
  implies that $\cC(\theta,p) = 2\, \cC(\wh\theta,p-a) - 2\,
  \cC(\wh\theta,p-a+1)$.  This establishes (iv) and completes the
  proof.
\end{proof}

Let $\theta$ be a rim.  A {\em KLG-tableau\/} of shape $\theta$ is a
labeling of the boxes of $\theta$ with elements from the ordered set
$\{1'<1<2'<2<\cdots\}$ such that (i) each row of $\theta$ is strictly
increasing from left to right; (ii) each column of $\theta$ is
strictly increasing from top to bottom; (iii) each box containing an
unprimed integer must be larger than or equal to all boxes southwest
of it, (iv) each box containing a primed integer must be smaller than
or equal to all boxes southwest of it, and (v) no diagonal box
contains a primed integer.  If $\theta$ is not a rim, then there are
no KLG-tableaux of shape $\theta$.  The content of a KLG-tableau is
the set of integers $i$ such that some box contains $i$ or $i'$.  The
proof of the following corollary is similar to
Corollary~\ref{cor:pieriBtab} and left to the reader.

\begin{cor}\label{cor:pieriCtab}
  The constant $c^\nu_{\la,p}$ for $K(\LG(n,2n))$ is equal to
  $(-1)^{|\nu/\la|-p}$ times the number of KLG-tableaux of shape
  $\nu/\la$ with content $\{1,2,\dots,p\}$.
\end{cor}

It would be very interesting to extend this corollary to a full
Littlewood-Richard\-son rule for all the structure constants
$c^\nu_{\la\mu}$ of $K(\LG(n,2n))$.

\begin{example}\label{ex:lgsym}
  Let $X = \LG(2,4)$, $\la = \mu = (1)$, and $\nu = (2,1)$.  Then
  $c^\nu_{\la\mu} = -1 \neq 0 = c^{\mu^\vee}_{\la,\nu^\vee}$.
\end{example}

\begin{example}
  Let $\la = (6,4,1)$ and $\nu = (7,6,3,1)$.  Then the constant
  $c^\nu_{\la,5} = -9$ for $K(\LG(7,14))$ is counted by the
  KLG-tableaux displayed below.  Notice that the lower-left box of
  $\nu/\la$ is a diagonal box.\medskip

  $\tableau{10}{&&&{1'}\\&&{1'}&{5}\\{2'}&{4}\\{3}}$ \ \ \ \ 
  $\tableau{10}{&&&{1'}\\&&{2'}&{5}\\{2'}&{4}\\{3}}$ \ \ \ \ 
  $\tableau{10}{&&&{1'}\\&&{2'}&{5}\\{3'}&{5}\\{4}}$ \ \ \ \ 
  $\tableau{10}{&&&{1'}\\&&{2'}&{5}\\{3'}&{4}\\{4}}$ \ \ \ \ 
  $\tableau{10}{&&&{1'}\\&&{2'}&{5}\\{3'}&{4}\\{3}}$\medskip

  $\tableau{10}{&&&{1'}\\&&{4}&{5}\\{1'}&{3}\\{2}}$ \ \ \ \ 
  $\tableau{10}{&&&{1'}\\&&{4}&{5}\\{2'}&{4}\\{3}}$ \ \ \ \ 
  $\tableau{10}{&&&{1'}\\&&{4}&{5}\\{2'}&{3}\\{3}}$ \ \ \ \ 
  $\tableau{10}{&&&{1'}\\&&{4}&{5}\\{2'}&{3}\\{2}}$
\end{example}


\input{bibliography}

\end{document}

%% file: preamble.tex
\newtheorem{thm}{Theorem}[section] 
\newtheorem{lemma}[thm]{Lemma}
\newtheorem{prop}[thm]{Proposition}
\newtheorem{cor}[thm]{Corollary}

\theoremstyle{definition}

\newtheorem{example}[thm]{Example}

\def\C{{\mathbb C}}
\def\Z{{\mathbb Z}}
\def\N{{\mathbb N}}
\def\P{{\mathbb P}}
\DeclareMathOperator{\Gr}{Gr}
\DeclareMathOperator{\Fl}{Fl}
\DeclareMathOperator{\LG}{LG}

\DeclareMathOperator{\IF}{IF}
\DeclareMathOperator{\OG}{OG}
\DeclareMathOperator{\OF}{OF}

\DeclareMathOperator{\QH}{QH}

\DeclareMathOperator{\cTor}{{\it Tor}}

\def\la{\lambda}
\def\bull{{\scriptscriptstyle\bullet}}
\def\ssm{\smallsetminus}
\def\wh{\widehat}
\def\wt{\widetilde}
\def\wb{\overline}
\def\cE{{\mathcal E}}
\def\cF{{\mathcal F}}
\def\cO{{\mathcal O}}
\def\cS{{\mathcal S}}
\def\op{\mathrm{op}}
\def\cA{{\mathcal A}}
\def\cB{{\mathcal B}}
\def\cC{{\mathcal C}}
\def\pt{\text{point}}
\newcommand{\bspan}[1]{\langle #1 \rangle}
\newcommand{\euler}[1]{\chi_{_{#1}}}
\newcommand{\pic}[2]{\includegraphics[scale=#1]{#2}}

\newcommand{\ignore}[1]{}


%% file: psfrag.tex
\psfrag{la}{$\la$}
\psfrag{mu}{$\mu$}
\psfrag{a}{$a$}
\psfrag{b}{$b$}
\psfrag{lap}{$\la'$}
\psfrag{mup}{$\mu'$}
\psfrag{thp}{$\theta'$}
\psfrag{lapp}{$\la''$}
\psfrag{mupp}{$\mu''$}
\psfrag{thpp}{$\theta''$}
\psfrag{north-east arm}{North-east arm}
\psfrag{whth}{$\wh\theta$}


%% file: bibliography.tex
\providecommand{\bysame}{\leavevmode\hbox to3em{\hrulefill}\thinspace}
\providecommand{\MR}{\relax\ifhmode\unskip\space\fi MR }
\providecommand{\MRhref}[2]{%
  \href{http://www.ams.org/mathscinet-getitem?mr=#1}{#2}
}
\providecommand{\href}[2]{#2}